\documentclass{amsart}

\usepackage[utf8]{inputenc}
\usepackage[T1]{fontenc}
\usepackage{lmodern}
\usepackage[english]{babel}
\usepackage{tikz}
\usepackage{tikz-cd}
\usetikzlibrary{positioning,arrows,shapes}
\usepackage{hyperref}
\usepackage[english,capitalize]{cleveref}
\usepackage{amssymb}
\usepackage{enumitem}
\setenumerate{label=(\roman*)}

\theoremstyle{plain}
\newtheorem{prop}{Proposition}[section]
\newtheorem{theorem}{Theorem}
\newtheorem*{theorem*}{Theorem}

\newtheorem*{question*}{Question}
\newtheorem*{lemma*}{Lemma}
\newtheorem{lemma}[prop]{Lemma}
\newtheorem{cor}[prop]{Corollary}
\newtheorem*{cor*}{Corollary}
\newtheorem{conj}[theorem]{Conjecture}
\newtheorem*{conj*}{Conjecture}

\theoremstyle{definition}
\newtheorem{defi}[prop]{Definition}
\newtheorem*{defi*}{Definition}
\newtheorem{nota}[prop]{Notation}
\newtheorem*{nota*}{Notation}
\newtheorem{rem}[prop]{Remark}
\newtheorem*{rem*}{Remark}
\newtheorem{ex}[prop]{Example}
\newtheorem{cons}[prop]{Construction}

\newcommand{\mm}{\,\middle|\,}
\newcommand{\g}{\mathfrak{g}}

\newcommand{\RR}{\mathbb{R}}
\newcommand{\ZZ}{\mathbb{Z}}

\renewcommand{\P}{\mathcal{P}}
\newcommand{\Q}{\mathcal{Q}}
\renewcommand{\L}{\mathcal{L}}
\renewcommand{\O}{\mathcal{O}}
\newcommand{\C}{\mathcal{C}}
\newcommand{\D}{\mathcal{D}}
\newcommand{\V}{\mathcal{V}}
\newcommand{\pre}{\mathrm{pre}}
\newcommand{\GT}{\mathcal{GT}}

\newcommand{\SL}{\mathrm{SL}}
\newcommand{\SO}{\mathrm{SO}}
\newcommand{\Sp}{\mathrm{Sp}}
\newcommand{\Spin}{\mathrm{Spin}}

\newcommand{\An}{\mathsf{A}_n}
\newcommand{\Bn}{\mathsf{B}_n}
\newcommand{\Cn}{\mathsf{C}_n}
\newcommand{\Dn}{\mathsf{D}_n}
\newcommand{\EE}{\mathsf{E}_6}
\newcommand{\EEE}{\mathsf{E}_7}
\newcommand{\EEEE}{\mathsf{E}_8}
\newcommand{\FF}{\mathsf{F}_4}
\newcommand{\GG}{\mathsf{G}_2}

\let\vert\relax
\DeclareMathOperator{\vert}{vert}
\DeclareMathOperator{\Lie}{Lie}
\DeclareMathOperator{\sgn}{sgn}

\newcommand{\udots}{\mathrel{\reflectbox{\ensuremath{\ddots}}}}

	\title{A Diagrammatic Approach to String Polytopes}

\author{Christian Steinert}

\address{Mathematical Institute, Faculty of Mathematics and Natural Sciences, University of Cologne; Chair for Algebra and Representation Theory, RWTH Aachen University}

\email{steinert@art.rwth-aachen.de}

\allowdisplaybreaks

\begin{document}
	
	\begin{abstract}\noindent
		We prove that for every complex classical group $G$ the string polytope associated to a special reduced decomposition and any dominant integral weight $\lambda$ will be a lattice polytope if and only if the highest weight representation of the Lie algebra of $G$ with highest weight $\lambda$ integrates to a representation of $G$ itself. This affirms an earlier conjecture and shows that every partial flag variety of a complex classical group admits a flat projective degeneration to a Gorenstein Fano toric variety.
	\end{abstract}
	
	\maketitle

\section*{Introduction}

Rational convex polytopes play a crucial role in representation theory for different reasons. First and foremost, there are many polytopes whose lattice points parameterize bases of highest weight representations. The first such polytope for $\SL_{n+1}$ was defined by Gelfand and Tsetlin in \cite{GT}. Berenstein and Zelevinsky defined analogous {\em Gelfand-Tsetlin polytopes} for all classical Lie algebras. It should be noted that the corresponding $\Sp_{2n}$-polytopes have been constructed before by Zhelobenko in \cite{Zh}. These constructions lead to the definition of {\em string polytopes} by Littelmann in \cite{L}. In contrast to the previous polytopes, these string polytopes give many different parameterizations of bases depending on the choice of a reduced decomposition of the longest word of the Weyl group. They are the main objects of our studies.

It should be noted that there exists a different {\em string polytope} by Nakashima and Zelevinsky \cite{NZ} that will not be discussed in this paper. Other famous polytopes have been defined by Lusztig in \cite{Lu}, by Feigin, Fourier and Littelmann in \cite{FFL1} and \cite{FFL2} as well by Gornitskii in \cite{G1} and \cite{G2}. The latter four polytopes are based on a conjecture by Vinberg.

Although all of these polytopes are united in the fact that their lattice polytopes give parameterizations of bases of highest weight representations, they omit different combinatorial properties. Even the most basic question whether a certain polytope is a lattice polytope, yields different answers for the aforementioned polytopes. In fact, even for string polytopes associated to {\em nice} decompositions (see \cite[Section 4]{L}), this question is highly non-trivial. 

It has been conjectured before by Alexeev and Brion that for $\SL_{n+1}$ every string polytope will be a lattice polytope independent of the reduced decomposition and the weight (see \cite[Conjecture 5.8]{AB}). Explicit calculations show that this conjecture holds for $n \leq 4$. However, in \cite[Example 5.5]{S2} we were able to provide counterexamples for $\SL_6$ and $\SL_7$ that can be extended to arbitrary rank. 

Even worse, one of the very few integrality results for string polytopes \cite[Theorem 4.5]{AB} is faulty (see \cite[Remark 5.9]{S2}). However, based on calculations we were able to state the following conjecture (see \cite[Conjecture 5.10]{S2}). For simplicity, we will say that the nice decompositions that have been constructed in \cite[Sections 5, 6 and 7]{L} are {\em standard}.

\begin{conj}\label{conj:string}
	Let $G$ be a complex classical group, let $\lambda \in \Lambda^+$ and let $\underline{w_0}^\mathrm{std}$ be the standard reduced decomposition of the longest word of the Weyl group of $G$ as stated in \cite{L}. Then $Q_{\underline{w_0}^\mathrm{std}}(\lambda)$ is a lattice polytope if and only if one of the following conditions hold.
	\begin{enumerate}
		\item $G=\SL_{n+1}$,
		\item $G=\SO_{2n+1}$ and $\langle \lambda, \alpha_n^\vee\rangle \in 2\ZZ$,
		\item $G=\Sp_{2n}$ or
		\item $G=\SO_{2n}$ and $\langle \lambda, \alpha_{n-1}^\vee\rangle + \langle \lambda, \alpha_n^\vee \rangle \in 2\ZZ$ or $n<4$.
	\end{enumerate}
\end{conj}

This paper is dedicated to proving said conjecture. By standard results from representation theory (see for example \cite[Chapter 10, Theorem 6.1]{P} and \cite[Chapter 11, Theorem 6.6]{P}) the following formulation is equivalent.

\begin{theorem}\label{thm:main}
	Let $G$ be a complex classical group and $\lambda$ a dominant integral weight of its Lie algebra $\Lie G$. Then the standard string polytope $\Q_{\underline{w_0}^\mathrm{std}}(\lambda)$ (in the sense of \cite{L}) is a lattice polytope if and only if the $(\Lie G)$-representation on $V(\lambda)$ integrates to a representation of $G$.
\end{theorem}

It should be noted that this result has been proved before in types $\An$ and $\Cn$ since the corresponding string polytope are unimodularly equivalent to the Gelfand-Tsetlin polytope. This polytope can be realized as a {\em marked order polytope} as defined by Ardila, Bliem and Salazar in \cite{ABS}, which readily yields the claim. Alternatively, a result on marked order polytopes by Fang and Fourier could also be used \cite{FF}.

However, these results do not prove the orthogonal cases (completely). In type $\Bn$ one only gets one implication of \cref{thm:main}. In type $\Dn$ there exists no affine bijection from the string polytope to the Gelfand-Tsetlin polytope. There exists only a piecewise affine bijection (see \cite[Section 7]{L}) which need not preserve vertices. Additionally the Gelfand-Tsetlin polytope in type $\Dn$ is not realized as a marked order polytope.

Our proof will give a visual tool to classify vertices of marked order polytopes via directed graphs that we will call {\em identity diagrams}. This classification proves the claim of \cref{thm:main} in types $\An$, $\Bn$ and $\Cn$. The concrete statement is the following.

\begin{theorem}\label{thm:markedorder}
	A point in a marked order polytope is a vertex if and only if each connected component in its identity diagram contains a marked element.
\end{theorem}

This result is not new. Rather, it is a graphical version of a result by Pegel \cite[Proposition 3.2]{Pe}. But we will reprove it to give an easier understanding of our work in $\Dn$.

For $\Dn$ we will construct slight deviations of the usual Gelfand-Tsetlin polytopes. These {\em tweaked Gelfand-Tsetlin polytopes} are slightly more complicated as they are realized in a bigger ambient space, creating artificial redundancies in the coordinates. However, we will be able to show that they are better suited for our purpose than the usual Gelfand-Tsetlin polytopes.

\begin{theorem}\label{thm:stringtweak}
	Let $G =\SO_{2n}$. For every dominant integral weight of the Lie algebra of $G$, there exists an affine bijection between $\RR^{n(n-1)}$ and an affine subspace of $\RR^{n^2+n-2}$ that sends the string polytope onto the tweaked Gelfand-Tsetlin polytope.
	
	Furthermore, a point in the tweaked Gelfand-Tsetlin polytope will correspond to a lattice point in the string polytope if and only if its coordinates are either completely contained in $\ZZ$ or completely contained in $\ZZ+\frac{1}{2}$.
\end{theorem}

These new polytopes will still not be marked order polytopes. But we will construct an analogue of identity diagrams that fulfill the same purpose. These {\em tweaked Gelfand-Tsetlin diagrams} give a visual tool to determine whether a given point in the tweaked Gelfand-Tsetlin polytope is a vertex although the criterion is slightly more complicated. Ultimately, this yields to the proof of the claim of \cref{thm:main} for type $\Dn$.

Together with our main result in \cite{S2}, this will imply the following.

\begin{theorem}\label{cor:reflexivestring}
	Let $G$ be a complex classical group and $\lambda$ a dominant integral weight. The standard string polytope $\Q_{\underline{w_0}^\mathrm{std}}(\lambda)$ (in the sense of \cite{L}) is a reflexive polytope after translation by a lattice vector if and only if $\lambda$ is the weight of the anticanonical line bundle over some partial flag variety $G/P$.
\end{theorem}

This is important because Batyrev proved in \cite{B} that reflexive polytopes are in one-to-one correspondence with Gorenstein Fano toric varieties. Additionally, Alexeev and Brion were able to construct toric degenerations of spherical varieties using string polytopes in \cite{AB}. So we will reach the following conclusion.

\begin{theorem}\label{thm:flag}
	Let $G$ be a complex classical group. Then every partial flag variety of $G$ admits a flat projective degeneration to a toric Gorenstein Fano variety.
\end{theorem}

This observation could lead to a potential application of Batyrev's mirror symmetry construction in \cite{B} to partial flag varieties.

The structure of this paper is as follows. Before proving \cref{thm:main}, we will recall the most important definitions regarding string and Gelfand-Tsetlin polytopes in \cref{sec:stringgt} and marked order polytopes in \cref{sec:markedorder}. We will then construct our identity diagrams in \cref{sec:identity} and prove \cref{thm:markedorder}, which yields to the proof of \cref{thm:main} for types $\An$, $\Bn$ and $\Cn$. As another potential application we will describe an algorithm that delivers all vertices of a given string polytope in \cref{sec:alg}. In \cref{sec:tweak,sec:tweakdiagram} we will construct our modified Gelfand-Tsetlin patterns in type $\Dn$ and their corresponding diagrams, which leads to the final proof of \cref{thm:main} in \cref{sec:d}. We will conclude our paper with proofs of \cref{cor:reflexivestring} and \cref{thm:flag} in \cref{sec:applications}.

\textit{Acknowledgments.} This paper is a rewritten excerpt from my PhD thesis \cite{S}. I am very grateful to my PhD advisor Peter Littelmann, whose continued support was essential in the studies leading to this work. I would also like to thank my PhD referees Ghislain Fourier and Kiumars Kaveh for their kind and helpful comments.

\section{String Polytopes and Gelfand-Tsetlin Polytopes}\label{sec:stringgt}

Let us start by recalling the definition of so called {\em (Generalized) Gelfand-Tsetlin Patterns} introduced by Berenstein and Zelevinsky in \cite{BZ1}. We will widely stick to the notation in \cite{L} although we will make slight adjustments.

\begin{nota}
	For two numbers $a,b\in\RR$ the inequality $a\geq b$ will be written graphically as
	\[ \begin{matrix}
	a&\\&b
	\end{matrix} \hspace{10pt}\text{ or } \hspace{10pt}
	\begin{matrix}
	&b\\a&
	\end{matrix}.\]
\end{nota}

We will now define Gelfand-Tsetlin patterns for all classical types.

\begin{defi}[Gelfand-Tsetlin Patterns in Type $\An$]
	Let $G=\SL_{n+1}$ and $\lambda = \sum_{i=1}^{n}\lambda_i \epsilon_i\in\Lambda^+$. A \textbf{Gelfand-Tsetlin pattern} of type $\lambda$ is a tuple $(y_{i,j})\in\RR^{\frac{n(n+1)}{2}}$, $1\leq i\leq j\leq n$, such that the coordinates fulfill the relations in \cref{fig:gta}.
	
	\begin{figure}[htb]\centering
		\caption{Inequalities of Gelfand-Tsetlin Patterns in type $\An$.}\label{fig:gta}
		\begin{tikzpicture}
		\matrix (m) [matrix of nodes,align=center,row sep={4ex,between origins},column sep={2.5em,between origins},minimum width=2em]
		{
			$\lambda_1$&&$\lambda_2$ &&$\ldots$&&$\ldots$&&$\ldots$ && $\lambda_n$&&$0$\\
			&$y_{1,1}$&&$y_{1,2}$&&$\ldots$&&$\ldots$&&$y_{1,n-1}$&&$y_{1,n}$&\\
			&&$y_{2,2}$&&$y_{2,3}$&&$\ldots$&&$y_{2,n-1}$&&$y_{2,n}$&&\\
			&&&$\ddots$&&$\ddots$&&$\udots$&&$\udots$&&&\\
			&&&&$y_{n-2,n-2}$&&$y_{n-2,n-1}$&&$y_{n-2,n}$&&&&\\
			&&&&&$y_{n-1,n-1}$&&$y_{n-1,n}$&&&&&\\
			&&&&&&$y_{n,n}$&&&&&&&\\
		};
		\end{tikzpicture}\end{figure}
\end{defi}

\begin{rem}
	Notice that in Littelmann's definition of type $\An$ Gelfand-Tsetlin patterns, the top row would be included in the tuple $(y_{i,j})$ as the initial row (i.e. $y_{0,0}=\lambda_1$, \ldots, $y_{0,n-1}=\lambda_n$, $y_{0,n} = 0$). However, for fixed $\lambda$ this does only change the embedding of the pattern and not the pattern itself. So we adapted the definition to embed our patterns in a vector space whose dimension equals the number of positive roots of the algebraic group $G$. Additionally, these entries have a different character (we will call these entries a {\em marking} later on), so we would like to treat them separately.
\end{rem}

\begin{defi}[Gelfand-Tsetlin Patterns in Types $\Bn$ and $\Cn$]
	Let $G=SO_{2n+1}$ or $\Sp_{2n}$ and $\lambda = \sum_{i=1}^{n}\lambda_i \epsilon_i\in\Lambda^+$. A \textbf{Gelfand-Tsetlin pattern} of type $\lambda$ is a pair $(\mathbf{y}, \mathbf{z})$ of tuples $\mathbf{y} = (y_{i,j}) \in\RR^{\frac{n(n-1)}{2}}$, $2\leq i \leq n$, $i\leq j \leq n$, and $\mathbf{z} = (z_{i,j}) \in\RR^{\frac{n(n+1)}{2}}$, $1\leq i \leq n$, $i\leq j\leq n$, such that the coordinates fulfill the relations in \cref{fig:gtbc}.
	\begin{figure}[htb]\centering
		\caption{Inequalities of Gelfand-Tsetlin patterns in types $\Bn$ and $\Cn$.}\label{fig:gtbc}
		\begin{tikzpicture}
		\matrix (m) [matrix of nodes,align=center,row sep={4ex,between origins},column sep={2.5em,between origins},minimum width=2em]
		{
			$\lambda_1$&&$\lambda_2$ &&$\ldots$&&$\lambda_{n-2}$&&$\lambda_{n-1}$ && $\lambda_n$&&$0$\\
			&$z_{1,1}$&&$z_{1,2}$&&$\ldots$&&$z_{1,n-2}$&&$z_{1,n-1}$&&$z_{1,n}$&\\
			&&$y_{2,2}$&&$y_{2,3}$&&$\ldots$&&$y_{2,n-1}$&&$y_{2,n}$&&$0$\\
			&&&$z_{2,2}$&&$z_{2,3}$&&$\ldots$&&$z_{2,n-1}$&&$z_{2,n}$&\\
			&&&&$y_{3,3}$&&$y_{3,4}$&&$\ldots$&&$y_{3,n}$&&$0$\\
			&&&&&$z_{3,3}$&&$z_{3,4}$&&$\ldots$&&$z_{3,n}$&\\
			&&&&&&$\ddots$&&$\ddots$&&$\ddots$&&&\\
			&&&&&&&$z_{n-2,n-2}$&&$z_{n-2,n-1}$&&$z_{n-2,n}$&\\
			&&&&&&&&$y_{n-1,n-1}$&&$y_{n-1,n}$&&$0$\\
			&&&&&&&&&$z_{n-1,n-1}$&&$z_{n-1,n}$&\\
			&&&&&&&&&&$y_{n,n}$&&$0$\\
			&&&&&&&&&&&$z_{n.n}$&\\
			&&&&&&&&&&&&$0$\\
		};
		\end{tikzpicture}\end{figure}
\end{defi}

\begin{rem}
	Notice that in a Gelfand-Tsetlin pattern $(\mathbf{y},\mathbf{z})$ of type $\Bn$ or $\Cn$ the first row as well as the zeroes in the last column are not actually part of the tuple $(\mathbf{y},\mathbf{z})$. In Littelmann's definition, the first row would be included as $y_{1,1} = \lambda_1$, \ldots, $y_{1,n} = \lambda_n$. The reasons for our change of definition are the same as in type $\An$. Otherwise we want to stick with his notation, which yields to the awkward fact that our tuple $(\mathbf{y})$ starts with the index $i = 2$. However, we will not need these indices explicitly, so this should not become a problem. 
\end{rem}

\begin{defi}[Gelfand-Tsetlin Patterns in Types $\Dn$]
	Let $G=\SO_{2n}$ and $\lambda = \sum_{i=1}^{n}\lambda_i \epsilon_i\in\Lambda^+$. A \textbf{Gelfand-Tsetlin pattern} of type $\lambda$ is a pair $(\mathbf{y}, \mathbf{z})$ of tuples $\mathbf{y} = (y_{i,j}) \in\RR^{\frac{n(n-1)}{2}}$, $2\leq i \leq n$, $i\leq j\leq n$, and $\mathbf{z} = (z_{i,j}) \in\RR^{\frac{n(n-1)}{2}}$, $1\leq i \leq n-1$, $i\leq j\leq n-1$, such that 
	\begin{align*}\label{eq:dn}
	z_{1,n-1} &\leq \lambda_n + y_{2,n} + \min\{ \lambda_{n-1}, y_{2,n-1}\},\\
	z_{i,n-1} &\leq y_{i,n} + y_{i+1,n} + \min\{ y_{i,n-1}, y_{i+1,n-1}\} \text{ for all } 2\leq i\leq n-2,\\z_{n-1,n-1} &\leq y_{n-1,n} + y_{n,n} + y_{n-1,n-1},
	\end{align*}
	and the coordinates fulfill the relations in \cref{fig:gtd}.
	\begin{figure}[htb]\centering
		\caption{Inequalities of Gelfand-Tsetlin patterns in type $\Dn$.}\label{fig:gtd}
		\begin{tikzpicture}
		\matrix (m) [matrix of nodes,align=center,row sep={4ex,between origins},column sep={2.5em,between origins},minimum width=2em]
		{
			$\lambda_1$&&$\lambda_2$ &&$\ldots$&&$\lambda_{n-3}$&&$\lambda_{n-2}$ && $\lambda_{n-1}$&&$\lambda_n$\\
			&$z_{1,1}$&&$z_{1,2}$&&$\ldots$&&$z_{1,n-3}$&&$z_{1,n-2}$&&$z_{1,n-1}$&\\
			&&$y_{2,2}$&&$y_{2,3}$&&$\ldots$&&$y_{2,n-2}$&&$y_{2,n-1}$&&$y_{2,n}$\\
			&&&$z_{2,2}$&&$z_{2,3}$&&$\ldots$&&$z_{2,n-2}$&&$z_{2,n-1}$&\\
			&&&&$y_{3,3}$&&$y_{3,4}$&&$\ldots$&&$y_{3,n-1}$&&$y_{3,n}$\\
			&&&&&$z_{3,3}$&&$z_{3,4}$&&$\ldots$&&$z_{3,n-1}$&\\
			&&&&&&$\ddots$&&$\ddots$&&$\ddots$&&&\\
			&&&&&&&$z_{n-3,n-3}$&&$z_{n-3,n-2}$&&$z_{n-3,n-1}$&\\
			&&&&&&&&$y_{n-2,n-2}$&&$y_{n-2,n-1}$&&$y_{n-1,n}$\\
			&&&&&&&&&$z_{n-2,n-2}$&&$z_{n-2,n-1}$&\\
			&&&&&&&&&&$y_{n-1,n-1}$&&$y_{n,n}$\\
			&&&&&&&&&&&$z_{n-1,n-1}$&\\
			&&&&&&&&&&&&$y_{n,n}$\\
		};
		\end{tikzpicture}\end{figure}
\end{defi}

\begin{rem}
	As before we deviate from Littelmann's notation by not including the initial row containing the $\lambda_j$ in our notion of a Gelfand-Tsetlin pattern. For unifying notation of the additional inequalities, it is understood that we mean $\lambda_j$ if we write $y_{1,j}$.
\end{rem}

One can see quite easily that the set of all Gelfand-Tsetlin patterns for a given weight is a polytope.

\begin{defi}
	Let $G$ be a complex classical group of type $\mathsf{X}_n$ and let $\lambda\in\Lambda^+$. Let $N$ denote the number of positive roots of $G$. The set of all possible Gelfand-Tsetlin patterns of type $\lambda$ is called the \textbf{Gelfand-Tsetlin polytope} $GT_{\mathsf{X}_n}(\lambda)\subseteq\RR^N$ of $\mathsf{X}_n$ and $\lambda$. We will sometimes omit the subscript $\mathsf{X}_n$ if this is clear from the context.
\end{defi}

Sometimes we might be interested to use the original definitions instead. So we introduce the following notation.

\begin{nota}
	An extended Gelfand-Tsetlin pattern $\hat x$ is a Gelfand-Tsetlin pattern $x$ together with the $\lambda_i$ in its top row.
\end{nota}

We will now recall the connection between Gelfand-Tsetlin polytopes and standard string polytopes. For that purpose we will define a non-standard terminology.

\begin{defi}
	Let $G$ be a complex classical group. A Gelfand-Tsetlin pattern $x$ is called \textbf{standard} if one of the following conditions hold.
	\begin{enumerate}
		\item $G=\SL_{n+1}$ and all coordinates of $\hat x$ are integral.
		\item $G=\SO_{2n+1}$, the $z_{1,n}, \ldots, z_{n,n}$ are in $\frac{1}{2}\ZZ$ and the other coordinates of $\hat x$ are either all integral or all are in $\frac{1}{2}+\ZZ$.
		\item $G=\Sp_{2n}$ and all coordinates of $\hat x$ are integral.
		\item $G=\SO_{2n}$ and the coordinates of $\hat x$ are either all integral or all are in $\frac{1}{2}+\ZZ$.
	\end{enumerate}
\end{defi}

The following is a combination of \cite[Corollary 5 and Corollary 7]{L}.

\begin{theorem}[Littelmann]\label{thm:stringgtabc}
	Let $G$ be a complex classical group of type $\mathsf{X}_n\neq\Dn$ and let $N$ denote the number of positive roots. For each dominant integral weight $\lambda = \sum_{i=1}^n\lambda_i\epsilon_i$ there exists an affine bijection $\phi_\lambda\colon \RR^N \to \RR^N$ such that $\phi_\lambda(\Q_{\underline{w_0}^\mathrm{std}}(\lambda)) = GT_{\mathsf{X}_n}(\lambda)$.
	
	Furthermore, $\phi_\lambda$ induces a bijection between the lattice points in $\Q_{\underline{w_0}^\mathrm{std}}(\lambda)$ and the standard $\mathsf{X}_n$-Gelfand-Tsetlin patterns of type $\lambda$.
\end{theorem}

\begin{rem}
	Interestingly, in type $\An$ Cho, Kim, Lee and Park gave a combinatorial classification of all reduced decompositions whose string polytope is unimodularly equivalent to the Gelfand-Tsetlin polytope in \cite{CKLP}.
\end{rem}

In type $\Dn$ the situation is more delicate as can be seen in \cite[Corollary 9]{L}. In this case the map $\phi_\lambda$ will only be piecewise affine.

\begin{theorem}[Littelmann]\label{thm:stringgtd}
	Let $G=\SO_{2n}$ and let $N$ denote the number of positive roots. For each dominant integral weight $\lambda = \sum_{i=1}^n\lambda_i\epsilon_i$ there exists a piecewise affine bijection $\phi_\lambda\colon \RR^N \to \RR^N$ such that $\phi_\lambda(\Q_{\underline{w_0}^\mathrm{std}}(\lambda)) = GT_{\mathsf{D}_n}(\lambda)$.
	
	Furthermore, $\phi_\lambda$ induces a bijection between the lattice points in $\Q_{\underline{w_0}^\mathrm{std}}(\lambda)$ and the standard $\Dn$-Gelfand-Tsetlin patterns of type $\lambda$.
\end{theorem}

The non-affineness will lead to difficulties in proving \cref{thm:main}, which we will overcome by defining new Gelfand-Tsetlin polytopes for type $\Dn$ in \cref{sec:tweakdiagram}.

\section{Marked Order Polytopes}\label{sec:markedorder}

We will now introduce generalized versions of Stanley's order polytopes. The definition of the order polytope associated to a poset is due to Stanley \cite{Sta2}. A {\em marking} on the poset lead to a generalization by Ardila, Bliem and Salazar in \cite{ABS}. These polytopes have been studied by Pegel \cite{Pe}, Fang and Fourier \cite{FF} and others.

\begin{defi}\label{defi:markedorder}
	Let $(P,\leq)$ be a finite poset, i.e. $P$ is a finite set with a partial order $\leq$ on $P$.
	\begin{enumerate}
		\item The \textbf{Hasse diagram} of $P$ is a directed graph whose set of nodes is $P$ and there is an arrow $p\to q$ whenever $p< q$ and there exists no $r$ with $p< r < q$.
		\item A \textbf{marking} on $P$ is a pair $(A, \lambda)$ where $A$ is a subset of $P$ containing all minimal and maximal elements of $P$ and $\lambda = (\lambda_a)_{a\in A}\in\RR^A$ is a real vector such that $\lambda_a\leq\lambda_b$ whenever $a\leq b$. The triplet $(P, A, \lambda)$ is called a \textbf{marked poset}. We will call the elements of $A$ \textbf{marked elements}.
		\item Let $(A,\lambda)$ be a marking on $P$. The \textbf{marked order polytope} $\O_{P,A}(\lambda)$ associated to $(P, A, \lambda)$ is defined as
		\[ \O_{P,A}(\lambda) := \left\{ x\in\RR^{P\setminus A} \,\,\middle|\,\, \begin{tabular}{l}
		$x_p \leq x_q \text{ for all } p\leq q,$\\
		$\lambda_a \leq x_p \text{ for all }a\leq p,$\\
		$x_p \leq \lambda_b \text{ for all } p\leq b$
		\end{tabular} \right\}.\]
		
	\end{enumerate}
\end{defi}

\begin{rem}
	The Gelfand-Tsetlin polytopes of types $\An$, $\Bn$ and $\Cn$ are marked order polytopes where the marking is given by the dominant integral weight and some zeroes. The Gelfand-Tsetlin polytopes of type $\Dn$ however are not marked order polytopes because of the additional four-term inequalities in their definition.
\end{rem}

The following theorem is due to Pegel, expanding results by Ardila, Bliem and Salazar \cite[Lemma 3.5]{ABS} as well as Fang and Fourier \cite[Corollary 2.2]{FF}. It can be found in \cite[Proposition 3.2]{Pe}.

\begin{theorem}[Pegel]\label{cor:markedorder}
	Let $(P, A, \lambda)$ be any marked poset. The coordinates of every vertex of the marked order polytope $\O_{P,A}(\lambda)$ must lie in the set $\{\lambda_a\mid a\in A \}$. Especially, $\O_{P,A}(\lambda)$ is a lattice polytope if $\lambda$ is integral.
\end{theorem}

This implies the following.

\begin{cor}\label{cor:abc}
	Let $G$ be a complex classical group of type $\mathsf{X}_n\neq \Dn$ and $\lambda=\sum_{i=1}^{n}\lambda_i\epsilon_i\in\Lambda^+$. Then the Gelfand-Tsetlin polytope $GT_{\mathsf{X}_n}(\lambda)$ is a lattice polytope if $\mathsf{X}_n = \An$ or $\Cn$ or if $\mathsf{X}_n = \Bn$ and $\langle\lambda, \alpha_n^\vee\rangle \in 2\ZZ$.
\end{cor}

\begin{proof}
	By \cref{cor:markedorder} it is clear that $GT_{\mathsf{X}_n}(\lambda)$ will be a lattice polytope if all coordinates of the marking vector are integral. The set of non-zero coordinates is precisely $\{\lambda_1, \ldots, \lambda_n\}$, so we need to know when the coefficients of a dominant integral weight written in the $\epsilon_i$ are integral. By \cite[Chapter 10, Theorem 6.1, and Chapter 11, Theorem 6.6]{P} this is always the case in types $\An$ and $\Cn$. In type $\Bn$ however we could get half-integral coefficients. To be more precise, each $\lambda_i$ can be written as the sum of some integers plus $\frac{\langle \lambda, \alpha_n^\vee\rangle}{2}$. So we see that the $\lambda_i$ are integers if and only if $\langle\lambda,\alpha_n^\vee\rangle$ is an even integer, which concludes the proof.
\end{proof}

This corollary would allow us to prove one implication of the claims of \cref{conj:string} and hence \cref{thm:main} for types $\An$, $\Bn$ and $\Cn$ via Littelmann's affine bijection from \cref{thm:stringgtabc}. The reason is that the inverse of Littelmann's map sends vertices to vertices and lattice points in the Gelfand-Tsetlin patterns (notice that these are a only a proper subset of the standard Gelfand-Tsetlin patterns in type $\Bn$) to lattice points of the standard string polytope.

Sadly, we cannot prove the {\em only-if}-part in type $\Bn$ directly and we simply cannot use these methods in the case $\Dn$. 
Firstly, the Gelfand-Tsetlin polytope in type $\Dn$ is not a marked order polytope. Secondly, since Littelmann's bijection of \cref{thm:stringgtd} is only piecewise affine, it need not send vertices to vertices. Some vertices could be send to non-vertices and vice-versa. 

From examples I reached the following conjecture which would at least solve this problem. But I could not find a proof.

\begin{conj}
	Let $\phi_\lambda$ be the map from \cref{thm:stringgtd}, i.e. the piecewise affine bijection with $\phi_\lambda(\Q_{\underline{w_0}^\mathrm{std}}(\lambda)) = GT_{\Dn}(\lambda)$. Then $\phi_\lambda$ induces a bijection between $\vert \Q_{\underline{w_0}^\mathrm{std}}(\lambda)$ and $\vert GT_{\Dn}(\lambda)$.
\end{conj}

So we will develop a tweaked version of Gelfand-Tsetlin patterns in type $\Dn$ that can be studied more easily. Additionally we will introduce a new method to classify vertices of these tweaked Gelfand-Tsetlin polytopes via diagrammatic combinatorics in \cref{sec:tweakdiagram}.

We will also apply these methods to the other classical types, thereby reproving \cref{cor:abc} and additionally the missing second implication of \cref{conj:string} in type $\Bn$.

\section{Identity Diagrams}\label{sec:identity}

We will state our definitions for arbitrary marked posets. The reductions to the Gelfand-Tsetlin cases $\An$, $\Bn$ and $\Cn$ are obvious.

\begin{defi}
	Let $(P,A,\lambda)$ be a marked poset and let $x\in\O_{P,A}(\lambda)$. The \textbf{identity diagram} $\D_{P,A}^\lambda(x)$ associated to $(P,A,\lambda,x)$ is a graph that contains all nodes and arrows of the Hasse diagram of $P$. Additionally, we draw an arrow $q \to p$ between two nodes $p$ and $q$ whenever there exists an arrow $p\to q$ in the Hasse diagram of $P$ and $x_p = x_q$ (if $p,q\neq A$) or $\lambda_p = x_q$ (if $p \in A$) or $x_p = \lambda_q$ (if $q\in A$).
\end{defi}

Whenever we draw these identity diagrams, for simplicity we will represent double arrows $p \rightleftarrows q$ by straight lines and omit single arrows.  From this practice we get the following non-standard terminology.

\begin{defi}
	Let $(P,A,\lambda)$ be a marked poset, let $x\in\O_{P,A}(\lambda)$ and let $\D_{P,A}^\lambda(x)$ be the associated identity diagram. A subset $\C$ of nodes is called \textbf{connected} if it is connected via double arrows, i.e. for any two nodes $p$ and $q$ in $\C$ there exists a (possibly empty) sequence $p_1, \ldots, p_t\in\C$ such that
	\[ p \rightleftarrows p_1 \rightleftarrows \ldots \rightleftarrows p_t \rightleftarrows q.\]
	The maximal (with respect to inclusion) connected subsets are called \textbf{connected components}.
\end{defi}

Additionally, when drawing identity diagrams for Gelfand-Tsetlin patterns we will represent the nodes corresponding to marked elements as follows. The zeros in the rightmost column will be drawn as small circles, while the nodes corresponding to the $\lambda_i$ in the first row will be drawn as small crosses. This change is made for easier readability as the following example shows.

\begin{ex}\label{ex:identitydiagram}
	Let $G=\Sp_8$ and $\lambda = 2\epsilon_1 + 2\epsilon_2+\epsilon_3+0\epsilon_4$. The pattern 		\begin{center}\begin{tikzpicture}
		\matrix (m) [matrix of nodes,align=center,row sep={2ex,between origins},column sep={1em,between origins},minimum width=2em]
		{
			$2$&&$2$&&$1$&&$0$\\
			&$2$&&$\frac{3}{2}$&&$1$&\\
			&&$2$&&$1$&&$0$\\
			&&&$\frac{3}{2}$&&$1$&\\
			&&&&$1$&&$0$\\
			&&&&&$0$&\\
			&&&&&&$0$\\
		};
		\end{tikzpicture}\end{center}
	admits the identity diagram and its visually more appealing drawing depicted in \cref{fig:c4}.
	\begin{figure}[htb]\centering
		\caption{Identity diagram and visually more appealing drawing of the Gelfand-Tsetlin pattern in \cref{ex:identitydiagram}.}\label{fig:c4}
		\begin{tikzpicture}[>=stealth, every node/.prefix style={draw=black,circle, fill, minimum size=.125cm, inner sep=0pt, outer sep=0pt}]
		\draw
		(0,0) node (1) {}
		(1,0) node (2) {}
		(2,0) node (3) {}
		(3,0) node (4) {}
		(0.5,-0.5) node (5) {}
		(1.5,-0.5) node (6) {}
		(2.5,-0.5) node (7) {}
		(1,-1) node (8) {}
		(2,-1) node (9) {}
		(3,-1) node (10) {}
		(1.5,-1.5) node (11) {}
		(2.5,-1.5) node (12) {}
		(2,-2) node (13) {}
		(3,-2) node (14) {}
		(2.5,-2.5) node (15) {}
		(3,-3) node (16) {};
		\draw[thick,->,bend left] 
		(4) edge (7)
		(7) edge (3)
		(3) edge (6)
		(6) edge (2)
		(2) edge (5)
		(5) edge (1)
		(10) edge (7)
		(7) edge (9) 
		(9) edge (6)
		(6) edge (8)
		(8) edge (5)
		(10) edge (12)
		(12) edge (9)
		(9) edge (11) 
		(11) edge (8)
		(14) edge (12)
		(12) edge (13) 
		(13) edge (11)
		(14) edge (15)
		(15) edge (13)
		(16) edge (15)
		(1) edge (5)
		(5) edge (2)
		(5) edge (8)
		(3) edge (7)
		(9) edge (7)
		(9) edge (12)
		(13) edge (12)
		(15) edge (14)
		(15) edge (16)
		;
		\draw[thick]
		(4,-1.5) node[minimum size=0cm] (l) {}
		(4.5,-1.5) node[minimum size=0cm] (m) {}
		(5,-1.5) node[minimum size=0cm] (r) {}
		(l) edge[bend left] (m)
		(m) edge[->,bend right] (r);
		\foreach \Point in {(6,0), (7,0), (8,0)}
		\draw \Point node[cross out,thick] {};
		\foreach \Point in {(9,0), (9,-1)}
		\draw \Point node[fill=white]  (\Point) {};
		\draw
		(9,-2) node[fill=white] (1) {}
		(9,-3) node[fill=white] (2) {};
		\foreach \Point in {(6.5,-0.5), (7.5,-0.5), (8.5,-0.5), (7,-1), (8,-1), (7.5,-1.5), (8.5,-1.5), (8,-2), (8.5,-2.5)}
		\draw \Point node {};
		\draw[thick]
		(6,0) 
		--  (6.5,-0.5) 
		-- (7,0)
		(6.5,-0.5) 
		-- (7,-1)
		(8,0)
		-- (8.5,-0.5)
		-- (8,-1)
		-- (8.5,-1.5)
		-- (8,-2)
		(1)
		-- (8.5,-2.5)
		-- (2);
		\end{tikzpicture}\end{figure}
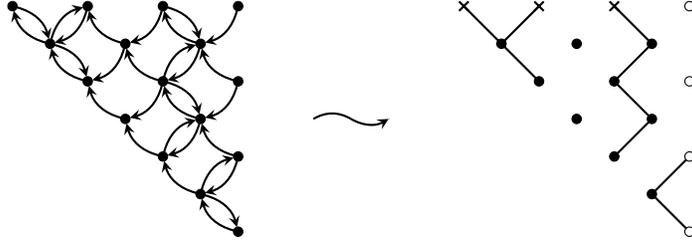
\end{ex}

These diagrams give us an easy way to {\em draw} vertices of marked order polytopes as the following result shows.

\begin{theorem*}[\cref{thm:markedorder}]
	Let $(P, A, \lambda)$ be a marked poset. A point $x\in\O_{P,A}(\lambda)$ is a vertex of the marked order polytope if and only if every connected component of the associated diagram $\D_{P,A}^\lambda(x)$ contains a marked element, i.e. an element of $A$.
\end{theorem*}

Although this theorem sounds quite technical, it is actually quite practical, as the following consequence shows.

\begin{cor}
	A point in a Gelfand-Tsetlin polytope is a vertex, if each entry of its Gelfand-Tsetlin pattern is equal to its upper left or upper right neighbor.
\end{cor}

To prove this result we use the following standard trick.

\begin{lemma}\label{lemma:vertices}
	Let $\P\subseteq\RR^d$ be a convex polytope. Then $x\in\P$ is a vertex of $\P$ if and only there does not exist a vector $v\in\RR^d$, $v\neq0$, such that $x+v \in \P$ and $x -v \in\P$.
\end{lemma}

We will now prove our theorem on the vertices of marked order polytopes.

\begin{proof}[Proof of \cref{thm:markedorder}]
	Let $x$ be a point in the marked order polytope $\O_{P,A}(\lambda)$. Suppose there exists a vector $v\in\RR^{P\setminus A}$ such that $x+v$ and $x-v$ both lie in the marked order polytope. Let $p,q\in P\setminus A$, $p\leq q$, be two nodes of the identity diagram $\D_{P,A}^\lambda(x)$ such that $p\leftrightarrows q$. Hence we know that $x_p = x_q$.
	
	Now since $x+v$ and $x-v$ lie in the marked order polytope we must have 
	\[ x_p+v_p \leq x_q+v_q \hspace{10pt}\text{ and }\hspace{10pt}x_p-v_p \leq x_q-v_q.\]
	This implies that $v_p = v_q$. Continuing this argument yields $v_p = v_q$ for any two nodes $p$ and $q$ of the identity diagram lying in the same connected component.
	
	Now let $\C$ be a connected component of $\D_{P,A}^\lambda(x)$. If $\C$ contains an element $a\in A$, we know that $x_p=\lambda_a$ for all $p\in\C$. Suppose that $\C$ is not completely contained in $A$. Then there exists a pair $p \leftrightarrows a$ with $p\in\C$ and $a \in A$. Without loss of generality let us assume that $p \leq a$. Then we know that 
	\[ x_p+v_p \leq \lambda_a\hspace{10pt}\text{ and }\hspace{10pt}x_p-v_p \leq \lambda_a,\]
	which implies that $v_p=0$. 
	
	In conclusion we see that $v_p=0$ for all $p\in P\setminus A$ such that the connected component of $p$ contains an element of $A$. By \cref{lemma:vertices} this implies that $x$ is a vertex whenever each connected component of $\D_{P,A}^\lambda(x)$ contains a marked element.
	
	For the other implication let us assume there exists a connected component $\C$ of $\D_{P,A}^\lambda(x)$ that does not contain any marked element. By definition of identity diagrams this means that $x_p < x_q$ for any $p\in\C$ and $q\in \P\setminus(\C\cup A)$ such that $p\leq q$. Additionally $x_p<\lambda_a$ for any $p\in\C$ and $a\in A$ such that $p\leq a$. Analogous statements hold if $p \geq q$ or $p\geq a$. Since the poset is finite we can find $\epsilon > 0$ such that 
	\begin{align*} 
	x_p &\pm \epsilon < x_q \text{ for all } p\in\C \text{ and }q\in \P\setminus(\C\cup A)\text{ such that } p\leq q,\\
	x_p &\pm \epsilon > x_q \text{ for all } p\in\C \text{ and }q\in \P\setminus(\C\cup A)\text{ such that } p\geq q,\\
	x_p &\pm \epsilon < \lambda_a \text{ for all } p\in\C \text{ and } a\in A\text{ such that } p\leq a,\\
	x_p &\pm \epsilon > \lambda_a \text{ for all } p\in\C \text{ and }a\in A\text{ such that } p\geq a.
	\end{align*}
	Consider the vector $v \in \RR^{P\setminus A}$ defined by
	\[ v_p := \begin{cases}
	\epsilon \hspace{10pt}\text{ if }p\in \C\\
	0 \hspace{10pt}\text{ else. }
	\end{cases}  \]
	Then it is clear that $x+v$ and $x-v$ both lie in $\O_{P,A}(\lambda)$. By \cref{lemma:vertices} this implies that $x$ is not a vertex of $\O_{P,A}(\lambda)$.
\end{proof}

Hence, we recover a proof of \cref{cor:markedorder}.

\begin{proof}[Proof of \cref{cor:markedorder}]
	Let $x\in\RR^{P\setminus A}$ be a vertex of $\O_{P,A}(\lambda)$. From \cref{thm:markedorder} we know that every connected component $\C$ of the identity diagram $\D_{P,A}^\lambda(x)$ of $x$ contains an element $a\in A$. By definition this means that that $x_p = \lambda_a$ for all $p\in\C$, hence every coordinate of $x$ must be equal to one of the $\lambda_a$.
\end{proof}

\section{Integrality of Standard String Polytopes in Types $\An$, $\Bn$ and $\Cn$}\label{sec:abc}

At first, we would like to make the connection between \cref{conj:string} and \cref{thm:main} explicit.


\begin{rem}
	If $G$ is simply connected\,---\,i.e. $G$ is of type $\An$ or $\Cn$\,---\,it is known (see for example \cite[Chapter 10, Theorem 6.1]{P}) that the irreducible representations of the Lie algebra $\g$ of $G$ are in one-to-one correspondence with the irreducible representations of $G$. So there are no further restrictions on $\lambda$ in these types.
	
	If $G$ however is not simply connected, i.e. $G=\SO_n$, it is known that not every irreducible representation of $\mathfrak{so}_n$ integrates to a representation of $\SO_n$. Instead, in general it integrates only to a representation of the spin group $\Spin_n$ as the universal covering of $\SO_n$. However, in some cases $V(\lambda)$ will still integrate to a representation of $\SO_n$. By \cite[Chapter 11, Theorem 6.6]{P} these cases are precisely the ones listed in \cref{conj:string}.
\end{rem}

We can now (re-)prove \cref{thm:main} in three types.

\begin{proof}[Proof of \cref{thm:main} in Types $\An$, $\Bn$ and $\Cn$]
	Let $G$ be of types $\An$, $\Bn$ or $\Cn$. By \cref{thm:stringgtabc} we know that the vertices of the standard string polytopes are in one-to-one correspondence with the vertices of the Gelfand-Tsetlin polytopes. We have seen that the Gelfand-Tsetlin polytopes are marked order polytopes. For $\lambda=\sum_{i=1}^{n}\lambda_i\epsilon_i\in \Lambda^+$ their marking is given by a vector whose coordinates are precisely the $\lambda_i$ and some zeros. 
	
	In types $\An$ and $\Cn$ we know that $\lambda_i\in\ZZ$ for every $1\leq i\leq n$. By \cref{cor:markedorder} we know that the vertices of the corresponding Gelfand-Tsetlin polytope have integral coordinates, so via Littelmann's map in \cref{thm:stringgtabc} they correspond to lattice points in the standard string polytope.
	
	In type $\Bn$ the same argument holds if $\langle \lambda,\alpha_n^\vee\rangle\in 2\ZZ$. However, if $\langle \lambda,\alpha_n^\vee\rangle$ is an odd integer, we know that $\lambda_i \in \frac{1}{2}+\ZZ$ for all $1 \leq i \leq n$. Now it is enough to notice that the pattern 
	
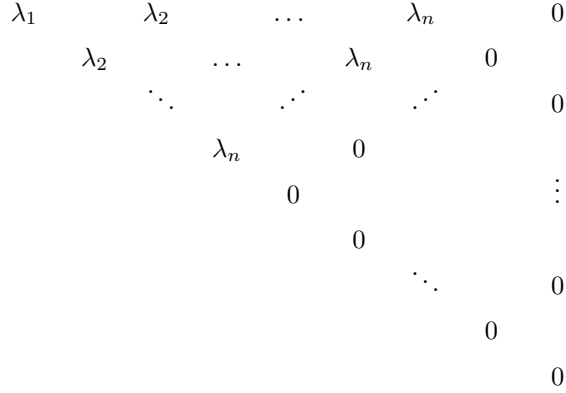
\begin{center}\begin{tikzpicture}
		\matrix (m) [matrix of nodes,align=center,row sep={4ex,between origins},column sep={2.5em,between origins},minimum width=2em]
		{
			$\lambda_1$ && $\lambda_2$ && $\ldots$ && $\lambda_n$ && $0$\\
			&$\lambda_2$&&$\ldots$&&$\lambda_n$&&$0$&\\
			&&$\ddots$&&$\udots$&&$\udots$&&$0$\\
			&&&$\lambda_n$&&$0$&&&\\
			&&&&$0$&&&&$\vdots$\\
			&&&&&$0$&&&\\
			&&&&&&$\ddots$&&$0$\\
			&&&&&&&$0$&\\
			&&&&&&&&$0$\\
		};
		\end{tikzpicture}\end{center}
	lies in $GT_{\Bn}(\lambda)$ for every $\lambda\in\Lambda^+$. Its identity diagram is drawn in \cref{fig:identityb}.
	
	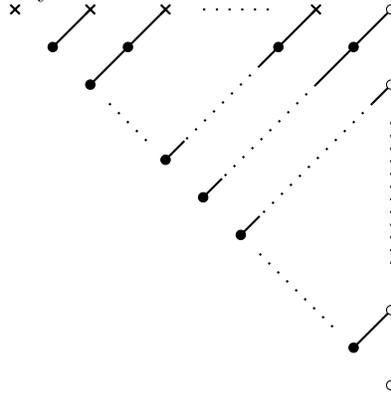
\begin{figure}[htb]\centering
		\caption{Identity diagram of the $\Bn$-Gelfand-Tsetlin pattern described in the proof of \cref{thm:main} with the usual drawing conventions for readability.}\label{fig:identityb}
		\begin{tikzpicture}[every node/.prefix style={draw=black,circle, fill, minimum size=.125cm, inner sep=0pt, outer sep=0pt}]
		\foreach \Point in {(0,0), (1,0), (2,0), (4,0)}
		\draw \Point node[cross out, thick] {};
		\draw 
		(5,0) node[fill=white] (1) {}
		(5,-1) node[fill=white] (2) {}
		(5,-4) node[fill=white] (3) {}
		(5,-5) node[fill=white] (4) {};
		\foreach \Point in {(0.5,-0.5) , (1.5,-0.5), (3.5,-0.5), (4.5,-0.5), (1,-1), (2,-2), (2.5,-2.5), (3,-3), (4.5,-4.5)}
		\draw \Point node {};
		\draw[thick,loosely dotted] (2.5,0) -- (3.5,0);
		\draw[thick,loosely dotted] (1.25,-1.25) -- (1.75,-1.75);
		\draw[thick,loosely dotted] (3.25,-3.25) -- (4.25,-4.25);
		\draw[thick,loosely dotted] (5,-1.5) -- (5,-3.5);		
		\draw[thick] (1,0) -- (0.5,-0.5);
		\draw[thick] (2,0) -- (1,-1);
		
		\draw[thick] (4,0) -- (3.25,-0.75);
		\draw[thick,loosely dotted] (3.25,-0.75) -- (2.25, -1.75);
		\draw[thick] (2.25,-1.75) -- (2,-2);
		\draw[thick] (1) -- (4,-1);
		\draw[thick,loosely dotted] (4,-1) -- (2.75, -2.25);
		\draw[thick] (2.75,-2.25) -- (2.5,-2.5);
		\draw[thick] (2) -- (4.75,-1.25);
		\draw[thick,loosely dotted] (4.75,-1.25) -- (3.25, -2.75);
		\draw[thick] (3.25,-2.75) -- (3,-3);
		\draw[thick] (3) -- (4.5,-4.5);
		\end{tikzpicture}\end{figure}
	
	We see that every connected component of the identity diagram contains a node corresponding to a marked element of the poset, hence this pattern must be a vertex of $GT_{\Bn}(\lambda)$ by \cref{thm:markedorder}.
	
	By construction and since $n>1$, the coordinate $y_{1,n} = \lambda_n$ of this pattern lies in $\frac{1}{2}+\ZZ$ while the coordinate $y_{2,n} = 0$ lies in $\ZZ$. This shows that the pattern is not standard, hence by \cref{thm:stringgtabc} its preimage under Littelmann's affine bijection is not a lattice point. So we found a non-integral vertex of the standard string polytope $\Q_{\underline{w_0}^\mathrm{std}}(\lambda)$ in type $\Bn$ for every $\lambda\in\Lambda^+$ such that $\langle \lambda,\alpha_n^\vee\rangle$ is odd, which concludes our proof in types $\An$, $\Bn$ and $\Cn$.
\end{proof}

\section{An Algorithmic Application}\label{sec:alg}

By studying identity diagrams abstractly, we can even give a complete classification of vertices of marked order polytope as the following construction in the Gelfand-Tsetlin case shows.

\begin{cons}
	We want to construct vertices of Gelfand-Tsetlin polytopes diagrammatically. Fix a dominant integral weight $\lambda$ and let us consider the poset $(P, A, \tilde{\lambda})$ such that $\O_{P,A}(\tilde{\lambda}) = GT(\lambda)$. Since the marking $\tilde{\lambda}$ is completely determined by the weight $\lambda$, we will just write $\lambda$ for both of them. 
	
	An important step in the construction will be the following {\em completion procedure}. Let $\mathcal{G}$ be a directed graph with node set $P$ that contains every arrow of the Hasse diagram of $P$. Assume that for every arrow $p \to q$ in $\mathcal{G}$ we have either an arrow $p\to q$ or an arrow $q\to p$ in the Hasse diagram. Additionally assume that $\mathcal{G}$ does not contain double arrows in the same direction. By this we mean that $p\leftrightarrows q$ is allowed but $p\rightrightarrows q$ is forbidden. We say that $\mathcal{G}$ is \textbf{complete} if the following two conditions hold.
	\begin{enumerate}
		\item Whenever there exists a set of arrows $p\leftrightarrows r \leftrightarrows q$ and $p \to s \to q$, there exists a set of arrows $p \leftarrow s \leftarrow q$ as well.
		\item Whenever there exists a sequence of arrows $a \to p_1 \to \ldots \to p_t \to b$ with $a, b \in A$ such that ${\lambda}_a = {\lambda}_b$, there exists a reverse sequence of arrows $a \leftarrow p_1 \leftarrow \ldots \leftarrow p_t \leftarrow b$ as well.
	\end{enumerate} 
	
	Notice that we can complete any graph with the mentioned assumptions by repeatedly adding new arrows\,---\,but only those that are strictly necessary\,---\,until the graph is complete. Of course, we might have to check every set of arrows repeatedly since we are constantly introducing new arrows in this process. However, since we will never produce a double arrow $p \rightrightarrows q$ and $P$ is finite, this algorithm will eventually stop. Additionally, the completion will be unique, i.e. it does not depend on the order in which we check for and\,---\,if necessary\,---\,add arrows.
	
	Coming back to our Gelfand-Tsetlin patterns, we can now describe a process to construct vertices of Gelfand-Tsetlin polytopes.
	
	We start with the Hasse diagram of the corresponding poset. First of all, the graph might not be complete. So whenever there exists an arrow $a \to p \to b$ for some $a,b\in A$ and $p\in P$ we must check whether ${\lambda}_a = {\lambda}_b$. Whenever this is the case, we must add the two arrows $b\to p\to a$ to the Hesse diagram. The resulting graph might not be complete after this initial step, so finish the completion procedure. 
	
	Since we want to create a vertex, we must add more arrows. We can freely introduce new arrows $p \to q$ whenever there exists an arrow $q\to p$ and there does not already exist an arrow $p\to q$. However, we must not add an arrow between two nodes $p$ and $q$ if $p$ is connected (via a possibly empty sequence of double arrows) to an element $a\in A$ and $q$ is connected (via a possibly empty sequence of double arrows) to an element $b\in A$ such that ${\lambda}_a \neq {\lambda}_b$. After adding an arrow, we must always complete the graph. 
	
	We must keep adding new arrows until we can no longer legally add new arrows. At that point, every vertex will be connected (via a sequence of double arrows) to at least one marked element, i.e. every connected component of the resulting diagram will contain at least one marked element.
	
	It is clear that we will always reach this stage. But of course the resulting graph is not unique. By adding different arrows, we will in general terminate in a different graph.
	
	By construction, every terminal graph in our algorithm will be the identity graph of a Gelfand-Tsetlin pattern. Every coordinate $x_p$ of the pattern is given by 
	$x_p = {\lambda}_a$, where $a$ is a marked element in the connected component of $p$.
	
	Because of \cref{thm:markedorder} this pattern must be a vertex of the Gelfand-Tsetlin polytope $GT(\lambda)$. Additionally, we are able to reach every vertex of $GT(\lambda)$ by this procedure (although admittedly it might take some time).
\end{cons}

Let us apply this procedure in an example.

\begin{ex}
	Let $G=\SO_5$ and consider the weight $\lambda = \omega_2 = \frac{1}{2}\epsilon_1 + \frac{1}{2}\epsilon_2$. We will use our usual convention to not draw the arrows of the Hasse diagramm but remember their existence by careful positioning of the nodes. Then the Hasse diagram is drawn in the following way.
	
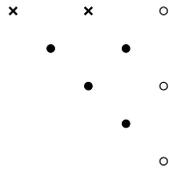
\begin{center}\begin{tikzpicture}[every node/.prefix style={draw=black, circle, fill, minimum size=.1cm, inner sep=0pt, outer sep=0pt}]
		\draw
		(0,0) node[cross out,thick] (1) {}
		(1,0) node[cross out,thick] (2) {}
		(2,0) node[fill=white] (3) {}
		(0.5,-0.5) node (4) {}
		(1.5,-0.5) node (5) {}
		(1,-1) node (6) {}
		(2,-1) node[fill=white] (7) {}
		(1.5,-1.5) node (8) {}
		(2,-2) node[fill=white] (9) {}
		;
		\end{tikzpicture}\end{center}
	
	As an initial step we must search for arrows $a\to p \to b$ with $a,b\in A$ such that the marking of $a$ and $b$ coincides. Since $\lambda_1 = \lambda_2 = \frac{1}{2}$ we have one such path in the upper left corner. Hence we must add arrows in the opposite direction. With our usual convention to draw $\leftrightarrows$ as straight lines we get the following diagram.
	
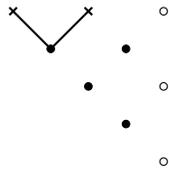
\begin{center}\begin{tikzpicture}[every node/.prefix style={draw=black, circle, fill, minimum size=.1cm, inner sep=0pt, outer sep=0pt}]
		\draw
		(0,0) node[cross out,thick] (1) {}
		(1,0) node[cross out,thick] (2) {}
		(2,0) node[fill=white] (3) {}
		(0.5,-0.5) node (4) {}
		(1.5,-0.5) node (5) {}
		(1,-1) node (6) {}
		(2,-1) node[fill=white] (7) {}
		(1.5,-1.5) node (8) {}
		(2,-2) node[fill=white] (9) {}
		;
		\draw[thick]
		(1) -- (4) -- (2);
		\end{tikzpicture}\end{center}
	
	Now we can add new arrows as opposites of already existing arrows. As an example, let us add an arrow from the middle vertex of the top row to its right bottom neighbor. Below is the resulting diagram and its completion.
	
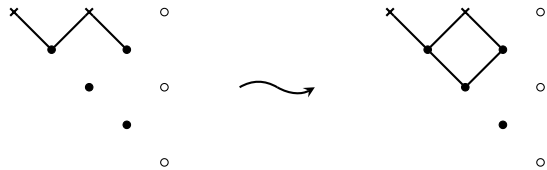
\begin{center}\begin{tikzpicture}[every node/.prefix style={draw=black, circle, fill, minimum size=.1cm, inner sep=0pt, outer sep=0pt}]
		\draw
		(0,0) node[cross out,thick] (1) {}
		(1,0) node[cross out,thick] (2) {}
		(2,0) node[fill=white] (3) {}
		(0.5,-0.5) node (4) {}
		(1.5,-0.5) node (5) {}
		(1,-1) node (6) {}
		(2,-1) node[fill=white] (7) {}
		(1.5,-1.5) node (8) {}
		(2,-2) node[fill=white] (9) {}
		;
		\draw[thick]
		(1) -- (4) -- (2) -- (5);
		\draw[>=stealth,thick]
		(3,-1) edge[bend left] (3.5,-1)
		(3.5,-1) edge[bend right,->] (4,-1);
		\draw
		(5,0) node[cross out,thick] (11) {}
		(6,0) node[cross out,thick] (22) {}
		(7,0) node[fill=white] (33) {}
		(5.5,-0.5) node (44) {}
		(6.5,-0.5) node (55) {}
		(6,-1) node (66) {}
		(7,-1) node[fill=white] (77) {}
		(6.5,-1.5) node (88) {}
		(7,-2) node[fill=white] (99) {}
		;
		\draw[thick]
		(11) -- (44) -- (22) -- (55) -- (66) -- (44);
		\end{tikzpicture}\end{center}
	
	For our next arrow we have three possible choices (all towards the bottom). Two possibilities give the same diagram after completion. The other possibility gives a different diagram. The two distinct complete diagrams are shown below.
	
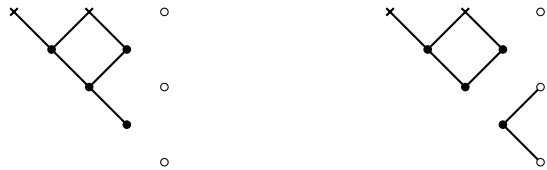
\begin{center}\begin{tikzpicture}[every node/.prefix style={draw=black, circle, fill, minimum size=.1cm, inner sep=0pt, outer sep=0pt}]
		\draw
		(0,0) node[cross out,thick] (1) {}
		(1,0) node[cross out,thick] (2) {}
		(2,0) node[fill=white] (3) {}
		(0.5,-0.5) node (4) {}
		(1.5,-0.5) node (5) {}
		(1,-1) node (6) {}
		(2,-1) node[fill=white] (7) {}
		(1.5,-1.5) node (8) {}
		(2,-2) node[fill=white] (9) {}
		;
		\draw[thick]
		(1) -- (4) -- (2) -- (5) -- (6) -- (4)
		(6) -- (8);
		\draw
		(5,0) node[cross out,thick] (11) {}
		(6,0) node[cross out,thick] (22) {}
		(7,0) node[fill=white] (33) {}
		(5.5,-0.5) node (44) {}
		(6.5,-0.5) node (55) {}
		(6,-1) node (66) {}
		(7,-1) node[fill=white] (77) {}
		(6.5,-1.5) node (88) {}
		(7,-2) node[fill=white] (99) {}
		;
		\draw[thick]
		(11) -- (44) -- (22) -- (55) -- (66) -- (44)
		(77) -- (88) -- (99);			
		\end{tikzpicture}\end{center}
	
	Both complete diagrams terminate the procedure. They correspond to the vertices $(\frac{1}{2}, \frac{1}{2}, \frac{1}{2}, \frac{1}{2})$ and $(\frac{1}{2}, \frac{1}{2}, \frac{1}{2}, 0)$ of $GT_{\mathsf{B}_2}(\omega_2)$.
	
	However, there are other possibilities by choosing a different arrow in the first addition. They lead to the following three complete diagrams.
	
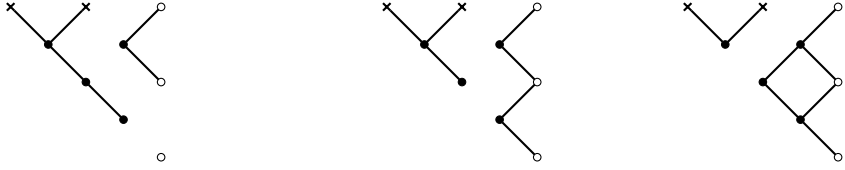
\begin{center}\begin{tikzpicture}[every node/.prefix style={draw=black, circle, fill, minimum size=.1cm, inner sep=0pt, outer sep=0pt}]
		\draw
		(0,0) node[cross out,thick] (1) {}
		(1,0) node[cross out,thick] (2) {}
		(2,0) node[fill=white] (3) {}
		(0.5,-0.5) node (4) {}
		(1.5,-0.5) node (5) {}
		(1,-1) node (6) {}
		(2,-1) node[fill=white] (7) {}
		(1.5,-1.5) node (8) {}
		(2,-2) node[fill=white] (9) {}
		;
		\draw[thick]
		(1) -- (4) -- (2)
		(4) -- (6) -- (8)
		(3) -- (5) -- (7)
		;
		\draw
		(5,0) node[cross out,thick] (11) {}
		(6,0) node[cross out,thick] (22) {}
		(7,0) node[fill=white] (33) {}
		(5.5,-0.5) node (44) {}
		(6.5,-0.5) node (55) {}
		(6,-1) node (66) {}
		(7,-1) node[fill=white] (77) {}
		(6.5,-1.5) node (88) {}
		(7,-2) node[fill=white] (99) {}
		;
		\draw[thick]
		(11) -- (44) -- (22)
		(44) -- (66)
		(33) -- (55) -- (77) -- (88) -- (99)
		;			
		\draw
		(9,0) node[cross out,thick] (111) {}
		(10,0) node[cross out,thick] (222) {}
		(11,0) node[fill=white] (333) {}
		(9.5,-0.5) node (444) {}
		(10.5,-0.5) node (555) {}
		(10,-1) node (666) {}
		(11,-1) node[fill=white] (777) {}
		(10.5,-1.5) node (888) {}
		(11,-2) node[fill=white] (999) {}
		;
		\draw[thick]
		(111) -- (444) -- (222)
		(333) -- (555) -- (777) -- (888) -- (999)
		(555) -- (666) -- (888)
		;		
		\end{tikzpicture}\end{center}
	
	Those are the identity diagrams of the vertices $( \frac{1}{2}, 0, \frac{1}{2}, \frac{1}{2})$, $(\frac{1}{2}, 0, \frac{1}{2}, 0)$ and $(\frac{1}{2}, 0, 0, 0)$ respectively. So we have found a visual way to calculate the $5$ vertices of $GT_{\mathsf{B}_2}(\omega_2)$.	
\end{ex}

\section{Tweaked Gelfand-Tsetlin Patterns }\label{sec:tweak}

Let us now consider the even orthogonal case. Before stating our new construction, let us first underline where the problems arise when trying to copy the previous proof of \cref{thm:main}.

In slight deviation of Littlemann's notation in \cite{L}, we will enumerate the simple roots of $\SO_{2n}$ as 
\[ \alpha_1 = \epsilon_1-\epsilon_2, \ldots, \alpha_{n-1} = \epsilon_{n-1} - \epsilon_n, \alpha_n = \epsilon_{n-1}+\epsilon_n. \]
Additionally we will consider the reduced decomposition
\[ \underline{w_0}^{\mathrm{std}} = (s_{n-1} s_n)(s_{n-2}s_{n-1} s_n s_{n-2}) \cdots (s_1 s_2\cdots s_{n-2} s_{n-1} s_n s_{n-2} \cdots s_2 s_1 ) \]
of the longest word of the Weyl group as the standard one. Notice that this does not completely correspond to Littelmann's standard decomposition since we swap the positions of the (commuting) reflections corresponding to $\epsilon_{n-1} - \epsilon_n$ and $\epsilon_{n-1}+\epsilon_n$. However, since the two reflections commute, the string polytopes will be the same after permutation of some coordinates. So we will sloppily say that these two polytopes are {\em the same}. 

We know that the polytope $\Q_{\underline{w_0}^\mathrm{std}}(\lambda)$ will be a subset of $\RR^{n(n-1)}$. We will denote the coordinates of a vector $a\in\RR^{n(n-1)}$ as 
\begin{align*}a &= (a_{n-1,n-1}, a_{n-1,n}, a_{n-2,n-2}, a_{n-2,n-1}, a_{n-2,n}, a_{n-2,n+1}, \ldots \\&\hspace{20pt}\ldots, a_{1,1}, a_{1,2}, \ldots, a_{1,n-2}, a_{1,n-1}, a_{1,n}, a_{1,n+1}, \ldots, a_{1,2n-2}).  \end{align*}
It is understood that $a_{i,j} = 0$ if any of the indices is outside of its allowed range. For every tuple $(a_{i,j})$ with $1\leq i\leq n-1$ and $i\leq j \leq 2n-1-i$ we will use the notation $\overline{a_{i,j}} := a_{i,2n-1-j}$.

We think of these coordinates as entries of the following triangle.
\begin{center}
	\begin{tikzpicture}
	\matrix (m) [matrix of math nodes,row sep=0.5em,column sep=0.5em,minimum width=2em] {
		a_{1,1}& a_{1,2}& \ldots& a_{1,n-2}& a_{1,n-1}& \overline{a_{1,n-1}} & \overline{a_{1,n-2}} & \ldots & \overline{a_{1,2}} & \overline{a_{1,1}}\\
		& a_{2,2}& \ldots& a_{2,n-2}& a_{2,n-1}& \overline{a_{2,n-1}} & \overline{a_{2,n-2}} & \ldots & \overline{a_{2,2}} & \\
		&& \ddots &  &  &  &  & \udots&&\\
		&&&a_{n-2,n-2}& a_{n-2,n-1}& \overline{a_{n-2,n-1}} & \overline{a_{n-2,n-2}}&&&\\
		&&&&a_{n-1,n-1}&\overline{a_{n-1,n-1}}&&&&\\
	};
	\end{tikzpicture}
\end{center}

Notice that our $(n-1)$-st column is Littelmann's $n$-th column and our $n$-th column is Littelmann's $(n-1)$-st column. The reason for this change is that the {$j$-th} column corresponds to the reflection $s_j$ for $j\leq n-2$. So it is more intuitive if the $(n-1)$-st row corresponds to the simple reflection $s_{n-1}$\,---\,and not $s_n$. However, these changes are rather cosmetic.

The following is a combination of \cite[Theorem 7.1 and Corollary 8]{L}.

\begin{theorem}[Littelmann]\label{thm:stringd}
	Let $G=\SO_{2n}$ and $\lambda=\sum_{i=1}^{n}\lambda_i\omega_i\in\Lambda^+$. A tuple $(a_{i,j}) \in \RR^{n(n-1)}$ is an element of $\Q_{\underline{w_0}^\mathrm{std}}(\lambda)$ if and only if the following two sets of conditions hold.
	\begin{center}
		$a_{i,i} \geq a_{i,i+1}, \geq \ldots \geq a_{i,n-2} \geq \left\{\begin{tabular}{c} $a_{i,n-1}$\\$\overline{a_{i,n-1}}$\end{tabular}\right\} \geq \overline{a_{i,n-2}} \geq \ldots \geq \overline{a_{i,i+1}} \geq \overline{a_{i,i}} \geq 0,$\\[.5\baselineskip]
		$a_{n-1,n-1} \geq 0, \hspace{10pt} \overline{a_{n-1,n-1}} \geq 0,$\\
	\end{center}
	for every $1\leq i\leq n-2$ and
	\begin{align*}
	a_{i,j} &\leq \lambda_j + a_{i,j+1} + \overline{a_{i,j+1}} - 2\overline{a_{i,j}} + \overline{a_{i,j-1}} \\&\hspace{20pt}+\sum_{k=1}^{i-1} (a_{k,j-1} - 2a_{k,j} + a_{k,j+1} + \overline{a_{k,j+1}} -2\overline{a_{k,j}} + \overline{a_{k,j-1}}),\\
	\overline{a_{i,j}} &\leq \lambda_j + \overline{a_{i,j-1}} + \sum_{k=1}^{i-1} (a_{k,j-1} - 2a_{k,j} + a_{k,j+1} + \overline{a_{k,j+1}} -2\overline{a_{k,j}} + \overline{a_{k,j-1}}),\\		
	a_{i,n-1} &\leq \lambda_{n-1} + \overline{a_{i,n-2}} + \sum_{k=1}^{i-1}(a_{k,n-2}-2a_{k,n-1}+\overline{a_{k,n-2}}),\\
	\overline{a_{i,n-1}} &\leq \lambda_{n} + \overline{a_{i,n-2}} + \sum_{k=1}^{i-1}(a_{k,n-2}-2\overline{a_{k,n-1}}+\overline{a_{k,n-2}}).
	\end{align*}
	for every $1\leq i\leq n-1$ and $i\leq j \leq n-2$.
\end{theorem}

Now we can describe an adapted version of Littelmann's piecewise affine map directly. Notice that we have to make slight adjustments because of our change of reduced decomposition. Let $\lambda=\sum_{i=1}^{n}\lambda_i\epsilon_i$. Notice the base change from $\omega_i$ to $\epsilon_i$ in contrast to \cref{thm:stringd}. Fix a point $(a_{i,j})\in\Q_{\underline{w_0}^\mathrm{std}}(\lambda)$. This point is sent via the piecewise affine bijection $\phi_\lambda$ to a Gelfand-Tsetlin pattern $x = (\mathbf{y}, \mathbf{z})$ in $\RR^{\frac{n(n-1)}{2}}\times\RR^{\frac{n(n-1)}{2}}$. By our convention the row index of $\mathbf{y} = (y_{i,j})$ starts with $i=2$. For easier notation we will set $y_{1,j} := \lambda_j$. Again, in our terminology this row is not actually part of the pattern $x$. The other rows can be computed reciprocally as
\begin{align*}
y_{i,j} &= y_{i-1,j} + a_{i-1,j-1} - a_{i-1,j} - \overline{a_{i-1,j}} + \overline{a_{i,j-1}} \text{ and }\\
y_{i,n} &= y_{i-1,n} + a_{i-1,n-1} - \overline{a_{i-1,n-1}}
\end{align*}
for every $2\leq i \leq n$ and $i\leq j \leq n-1$. For the $z$-coordinates we have the formulae
\begin{align*}
z_{i,j} &= y_{i,j} + \overline{a_{i,j-1}} - \overline{a_{i,j}},\\
z_{i,n-1} &= y_{i,n} + \min \{ a_{i,n-2} - \overline{a_{i,n-1}}, a_{i,n-1} - \overline{a_{i,n-2}} \} \text{ and }\\
z_{n-1,n-1} &= y_{i,n} + a_{n-1,n-1}
\end{align*}
for every $1\leq i \leq n-2$ and $i\leq j\leq n-2$.

So we see that the non-affine part appears in the coordinates $z_{i,n-1}$. Since $\phi_\lambda$ is a bijection we are not loosing any information when applying $\phi_\lambda$ but the minimum function makes it appear that way. 

Our goal now is to embed the Gelfand-Tsetlin pattern in a subspace of a larger vector space to keep track of both values $a_{i,n-2} - \overline{a_{i,n-1}}$ and $a_{i,n-1} - \overline{a_{i,n-2}}$. For that purpose we will introduce new coordinates $z_{i,n-1}^\uparrow$ and $z_{i,n-1}^\downarrow$ to replace the {\em bad} coordinate $z_{i,n-1}$.

For easier presentation we will use the following notation.

\begin{nota}
	Let $a$, $b$, $c$, $d$, $e$ and $f$ be some real numbers. We will write
	\begin{center}
		\begin{tikzpicture}
		\matrix (m) [matrix of math nodes, row sep=0em, column sep=0em]
		{ a && b\\&&\\&c&\\&d&\\e&&f\\};				
		\end{tikzpicture}
	\end{center}
	if the numbers fulfill the conditions
	\begin{align*}
	a\geq c\geq \left\{\begin{tabular}{c}$b$\\$f$\end{tabular}\right\}\hspace{10pt}\text{ and }\hspace{10pt} & c \leq a+b+f,\\
	e\geq d\geq \left\{\begin{tabular}{c}$b$\\$f$\end{tabular}\right\}\hspace{10pt}\text{ and }\hspace{10pt} & d \leq e+b+f.
	\end{align*}
	Notice and beware of the asymmetry in this notation! We do not for example require $a\geq d$!
\end{nota}

We can now define a modified version of $\Dn$-Gelfand-Tsetlin patterns.

\begin{defi}[Tweaked Gelfand-Tsetlin Patterns in Type $\Dn$]
	Let $G=\SO_{2n}$ and $\lambda = \sum_{i=1}^{n}\lambda_i \epsilon_i\in\Lambda^+$. A \textbf{tweaked Gelfand-Tsetlin pattern} of type $\lambda$ is a pair $(\mathbf{y}, \mathbf{z})$ of tuples $\mathbf{y} = (y_{i,j}) \in\RR^{\frac{n(n-1)}{2}}$, $2\leq i \leq n$, $i\leq j\leq n$, and 
	\[\mathbf{z} = (z_{1,1}, \ldots, z_{1,n-2}, z_{1,n-1}^\uparrow, z_{1,n-1}^\downarrow, \ldots, z_{n-2,n-2}, z_{n-2,n-1}^\uparrow, z_{n-2,n-1}^\downarrow, z_{n-1,n-1})\]such that
	\begin{align}\label{eq:tweak}
	y_{i,n-1}-y_{i+1,n-1} = z_{i,n-1}^\uparrow- z_{i,n-1}^\downarrow  \hspace{10pt} \text{ for all }1\leq i\leq n-2
	\end{align}
	and the coordinates fulfill the relations in \cref{fig:tweak}.
	\begin{figure}[htb]\centering
		\caption{Inequalities of tweaked Gelfand-Tsetlin patterns.}\label{fig:tweak}
		\begin{tikzpicture}
		\matrix (m) [matrix of nodes,align=center,row sep={4.5ex,between origins},column sep={2.5em,between origins},minimum width=2em]
		{
			$\lambda_1$&&$\lambda_2$ &&$\ldots$&&$\lambda_{n-3}$&&$\lambda_{n-2}$ && $\lambda_{n-1}$&&$\lambda_n$\\
			&$z_{1,1}$&&$z_{1,2}$&&$\ldots$&&$z_{1,n-3}$&&$z_{1,n-2}$&&\begin{tabular}{c}$z_{1,n-1}^\uparrow$\\$z_{1,n-1}^\downarrow$\end{tabular}&\\
			&&$y_{2,2}$&&$y_{2,3}$&&$\ldots$&&$y_{2,n-2}$&&$y_{2,n-1}$&&$y_{2,n}$\\
			&&&$z_{2,2}$&&$z_{2,3}$&&$\ldots$&&$z_{2,n-2}$&&\begin{tabular}{c}$z_{2,n-1}^\uparrow$\\$z_{2,n-1}^\downarrow$\end{tabular}&\\
			&&&&$y_{3,3}$&&$y_{3,4}$&&$\ldots$&&$y_{3,n-1}$&&$y_{3,n}$\\
			&&&&&$z_{3,3}$&&$z_{3,4}$&&$\ldots$&&\begin{tabular}{c}$z_{3,n-1}^\uparrow$\\$z_{3,n-1}^\downarrow$\end{tabular}&\\
			&&&&&&$\ddots$&&$\ddots$&&$\ddots$&&&\\
			&&&&&&&$z_{n-3,n-3}$&&$z_{n-3,n-2}$&&\begin{tabular}{c}$z_{n-3,n-1}^\uparrow$\\$z_{n-3,n-1}^\downarrow$\end{tabular}&\\
			&&&&&&&&$y_{n-2,n-2}$&&$y_{n-2,n-1}$&&$y_{n-1,n}$\\
			&&&&&&&&&$z_{n-2,n-2}$&&\begin{tabular}{c}$z_{n-2,n-1}^\uparrow$\\$z_{n-2,n-1}^\downarrow$\end{tabular}&\\
			&&&&&&&&&&$y_{n-1,n-1}$&&$y_{n,n}$\\
			&&&&&&&&&&&$z_{n-1,n-1}$&\\
			&&&&&&&&&&&&$y_{n,n}$\\
		};
		\end{tikzpicture}\end{figure}
	
	To simplify notation we will sometimes write $z_{n-1,n-1}^\uparrow$ for $z_{n-1,n-1}$.
\end{defi}

As in the usual definition, these patterns will define a polytope.

\begin{defi}
	Let $G=\SO_{2n}$ and $\lambda\in\Lambda^+$. The \textbf{tweaked Gelfand-Tsetlin polytope} $\widetilde{GT}(\lambda)$ is defined as the set of all tweaked Gelfand-Tsetlin patterns of type $\lambda$.
\end{defi}

The relation between usual Gelfand-Tsetlin patterns and tweaked Gelfand-Tsetlin patterns is given by the following observation.

Let $\V_\lambda$ denote the linear subspace of $\RR^{n(n-1)+(n-2)}$ defined by the relations in \cref{eq:tweak}.

\begin{theorem}\label{thm:tweak}
	For every $\lambda\in\Lambda^+$ there exists a bijection $\psi_\lambda\colon\V_\lambda\to\RR^{n(n-1)} $ given by $(z_{i,n-1}^\uparrow, z_{i,n-1}^\downarrow) \mapsto \min\{z_{i,n-1}^\uparrow, z_{i,n-1}^\downarrow\}$ for all $1\leq i\leq n-2$ and identity on the other coordinates. Its inverse is given by 
	\begin{align*}
	z_{i,n-1}^\uparrow &:= z_{i,n-1} + y_{i,n-1} - \min\{y_{i,n-1},y_{i+1,n-1}\}\\
	z_{i,n-1}^\downarrow &:= z_{i,n-1} + y_{i+1,n-1} - \min\{ y_{i,n-1},y_{i+1,n-1} \}
	\end{align*}
	for every $1\leq i\leq n-2$ and identity on the other coordinates. Furthermore, $\psi_\lambda$ induces a bijection between the tweaked Gelfand-Tsetlin patterns of type $\lambda$ and the usual Gelfand-Tsetlin patterns of type $\lambda$.
\end{theorem}

The proof is done by explicit calculation. It can be found in \cite[Theorem 6.7.5]{S}.
We can now state an analogue of \cref{thm:stringgtabc} for $\Dn$.

\begin{theorem*}[\cref{thm:stringtweak}]
	Let $G=\SO_{2n}$ and $\lambda=\sum_{i=1}^{n}\lambda_i\epsilon_i\in\Lambda^+$. The map $\widetilde{\phi}_\lambda:=\psi_\lambda^{-1}\circ\phi_\lambda\colon \RR^{n(n-1)} \to \V_\lambda$ is an affine bijection and $\widetilde{\phi}_\lambda(\Q_{\underline{w_0}^\mathrm{std}}(\lambda)) = \widetilde{GT}(\lambda)$.
	Furthermore, $a\in\RR^{n(n-1)}$ is a lattice point if and only if the coordinates of $\widetilde{\phi}_\lambda(a)$\,---\,including the first row $y_{1,j}= \lambda_j$\,---\,are either all integral or all are in $\frac{1}{2}+\ZZ$.
\end{theorem*}

\begin{proof}
	Since $\phi_\lambda$ and $\psi_\lambda^{-1}$ are piecewise affine bijections, the same holds true for $\widetilde{\phi}_\lambda$. Since $\phi_\lambda(\Q_{\underline{w_0}^\mathrm{std}}) = GT_{\Dn}(\lambda)$ by \cref{thm:stringgtd} and $\psi_\lambda(\widetilde{GT}(\lambda)) = GT_{\Dn}(\lambda)$ by \cref{thm:tweak}, we have $\widetilde{\phi}_\lambda(\Q_{\underline{w_0}^\mathrm{std}}(\lambda)) = \widetilde{GT}(\lambda)$. The claim on lattice points is clear from \cref{thm:stringgtd} and the definition of $\psi_\lambda$.
	
	It remains to show that $\widetilde{\phi}_\lambda$ and $\widetilde{\phi}_\lambda^{-1}$ are in fact affine. We only have to check the coordinates $z_{i,n-1}^\uparrow$ and $z_{i,n-1}^\downarrow$ since the map is affine in all other coordinates. Let $x = (\mathbf{y}, \mathbf{z})$ be the image of $(a_{i,j})$ under $\widetilde{\phi}_\lambda$. We calculate
	\begin{align*}
	z_{i,n-1}^\uparrow &= z_{i,n-1} + y_{i,n-1} - \min\{ y_{i,n-1}, y_{i+1,n-1} \}\\
	&= z_{i,n-1} - \min\{ 0, y_{i+1,n-1}-y_{i,n-1} \}\\
	&= y_{i,n} + \min \{ a_{i,n-2} - \overline{a_{i,n-1}}, a_{i,n-1} - \overline{a_{i,n-2}}\}\\
	&\hspace{20pt} - \min \{ 0, a_{i,n-2} - a_{i,n-1} - \overline{a_{i,n-1}}  + \overline{a_{i,n-2}} \}\\
	&= y_{i,n} +  a_{i,n-1} - \overline{a_{i,n-2}}.			
	\end{align*}
	This implies that $z_{i,n-1}^\uparrow$ is actually a linear combination of the coordinates of $(a_{i,j})$ plus $\lambda_n$. The same holds true for $z_{i,n-1}^\downarrow$ by analogous computation. Hence the map $\widetilde{\phi}_\lambda$ is affine, i.e. the concatenation of a linear map and a translation. Since the inverse of a linear map is linear and the inverse of a translation is a translation we know that $\widetilde{\phi}_\lambda^{-1}$ must be affine too. This concludes the proof. 
\end{proof}

\begin{ex}\label{ex:d3tweak}
	Let $G=\SO_6$ and let $\lambda=\sum_{i=1}^n\lambda_i\epsilon_i$. The image of the point $(a,b,c,d,e,f)\in\Q_{\underline{w_0}^\mathrm{std}}(\lambda)$ under the map $\widetilde{\phi}_\lambda$ is drawn in \cref{fig:d3tweak}.
\end{ex}

\begin{figure}[htb]\centering
	\caption{The point $\widetilde{\phi}_\lambda(a,b,c,d,e,f)$ from \cref{ex:d3tweak}.}\label{fig:d3tweak}
	\begin{tikzpicture}
	\matrix (m) [matrix of nodes,align=center,row sep={5ex,between origins},column sep={5em,between origins},minimum width=2em]
	{
		$\lambda_1$&&$\lambda_2$&&$\lambda_3$\\
		&$\lambda_1-e+f$&&\begin{tabular}{c}$\lambda_3+d-f$\\$\lambda_3+c-e$\end{tabular}&\\
		&&$\lambda_2+c-d-e+f$&&$\lambda_3+d-e$\\
		&&&$\lambda_3+a+d-e$&\\
		&&&&$\lambda_3+a-b+d-e$\\
	};
	\end{tikzpicture}
\end{figure}

\section{Tweaked Gelfand-Tsetlin Diagrams}\label{sec:tweakdiagram}

We will now define an analogue of identity diagrams of elements of marked order polytopes for tweaked Gelfand-Tsetlin patterns. For that purpose we want to define a poset that describes as much of the tweaked Gelfand-Tsetlin polytope as possible.

\begin{cons}[Tweaked Gelfand-Tsetlin Poset]
	Let $G=\SO_{2n}$. Let $\GT_n$ be the set of symbols 
	\begin{enumerate}
		\item $\xi_{i,j}$ for $1\leq i \leq n$ and $i\leq j \leq n$,
		\item $\zeta_{i,j}$ for $1\leq i \leq n-2$ and $i\leq j\leq n-2$,
		\item $\zeta_{i,n-1}^\uparrow$ and $\zeta_{i,n-1}^\downarrow$ for $1\leq i \leq n-2$, and
		\item $\zeta_{n-1,n-1}$.
	\end{enumerate}
	For easier notation we sometimes write $\zeta_{n-1,n-1}^\uparrow$ for $\zeta_{n-1,n-1}$. The ordering is given precisely by the inequalities in the definition of tweaked Gelfand-Testlin patterns that do not require addition or subtraction by substituting $\xi$ for $y$ and $\zeta$ for $z$.
\end{cons}

Notice that this describes almost all relations defining tweaked Gelfand-Tsetlin patterns. The ones missing are the four-term relations of type 
\begin{align*}
z_{i,n-1}^\uparrow - z_{i,n-1}^\downarrow &= y_{i,n-1} - y_{i+1,n-1} ,\\
z_{i,n-1}^\uparrow &\leq y_{i,n-1} + y_{i,n} + y_{i+1,n}\text{ and }\\
z_{i,n-1}^\downarrow &\leq y_{i+1,n-1} + y_{i,n} + y_{i+1,n}.
\end{align*}
We will return to them later.

\begin{ex}
	The Hasse diagram of the tweaked Gelfand-Tsetlin poset $\GT_4$ is depicted in \cref{fig:gt4}.
	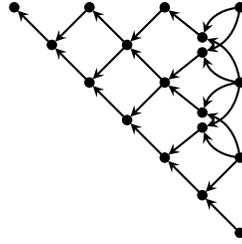
\begin{figure}[htb]\centering
		\caption{Hasse diagram of the tweaked Gelfand-Tsetlin poset $\GT_4$.}\label{fig:gt4}
		\begin{tikzpicture}[>=stealth, every node/.prefix style={draw=black,circle, fill, minimum size=.125cm, inner sep=0pt, outer sep=0pt},scale=1]
		\draw
		(0,0) node (1) {}
		(1,0) node (2) {}
		(2,0) node (3) {}
		(3,0) node (4) {}
		(0.5,-0.5) node (5) {}
		(1.5,-0.5) node (6) {}
		(2.5,-0.4) node (7) {}
		(2.5,-0.6) node (77) {}
		(1,-1) node (8) {}
		(2,-1) node (9) {}
		(3,-1) node (10) {}
		(1.5,-1.5) node (11) {}
		(2.5,-1.4) node (12) {}
		(2.5,-1.6) node (1212) {}
		(2,-2) node (13) {}
		(3,-2) node (14) {}
		(2.5,-2.5) node (15) {}
		(3,-3) node (16) {};
		\draw[thick,->] 
		(7) edge (3)
		(3) edge (6)
		(6) edge (2)
		(2) edge (5)
		(5) edge (1)
		(77) edge (9) 
		(9) edge (6)
		(6) edge (8)
		(8) edge (5)
		(12) edge (9)
		(9) edge (11) 
		(11) edge (8)
		(1212) edge (13) 
		(13) edge (11)
		(14) edge (15)
		(15) edge (13)
		(16) edge (15)
		;
		\draw[thick,->,bend right]
		(4) edge (7)
		(10) edge (7) 
		(10) edge (12)
		(14) edge (12)
		;
		\draw[thick,->,bend left]
		(4) edge (77)
		(10) edge (77) 
		(10) edge (1212)
		(14) edge (1212)
		;
		\end{tikzpicture}\end{figure}
\end{ex}

We would like to define an analogue of the marked order polytope. For this we will use the following non-standard definition.

\begin{defi}
	A \textbf{pseudo-marking} on a poset $P$ is a pair $(A,\lambda)$ where $A$ is a subset of $P$ and $\lambda = (\lambda_a)_{a\in A}\in \RR^A$ is a real vector such that $\lambda_a \leq \lambda_b$ whenever $a\leq b$. The triplet $(P,A,\lambda)$ is called a \textbf{pseudo-marked poset}. We will call the elements of $A$ \textbf{marked elements}.
\end{defi}

Notice that in contrast to \cref{defi:markedorder} we do not require that all minimal and all maximal elements of the poset are marked. As a consequence, the following definition might not give a polytope.

\begin{defi}
	Let $(P,A,\lambda)$ be a pseudo-marked poset. The \textbf{marked order polyhedron} $\O_{P,A}(\lambda)$ associated to $(P,A,\lambda)$ is defined as
	\[ \O_{P,A}(\lambda) := \left\{ x\in\RR^{P\setminus A} \,\,\middle|\,\, \begin{tabular}{l}
	$x_p \leq x_q \text{ for all } p\leq q,$\\
	$\lambda_a \leq x_p \text{ for all }a\leq p,$\\
	$x_p \leq \lambda_b \text{ for all } p\leq b$
	\end{tabular} \right\}.\]
\end{defi}

The following observation is clear by construction.

\begin{prop}\label{prop:gtorder}
	Let $G=\SO_{2n}$. Set $A := \{ \xi_{1,1},\ldots,\xi_{1,n} \}\subseteq \GT_n$. Let $\lambda\in\RR^A$ and denote by $\tilde{\lambda} := \sum_{i=1}^{n}\lambda_{\xi_{1,i}}\epsilon_i$ the associated weight (not necessarily dominant nor integral). Then
	\[ \widetilde{GT}(\tilde{\lambda}) := \O_{\GT_n,A}(\lambda)\cap\V_{\tilde{\lambda}}\cap\left\{\begin{tabular}{c}
	$z_{i,n-1}^\uparrow \leq y_{i,n-1} + y_{i,n} + y_{i+1,n}$\\ $z_{i,n-1}^\downarrow \leq y_{i+1,n-1} + y_{i,n} + y_{i+1,n} $
	\end{tabular}\,\,\middle|\,\,1\leq i\leq n-1\right\}.  \]
\end{prop}

Notice that we were a bit sloppy with our notation here since $\widetilde{GT}(\tilde{\lambda})$, $\O_{\GT_n,A}(\lambda)$ and $\V_{\tilde{\lambda}}$ are defined as subsets of $\RR^{\GT_n\setminus A}$. So we want to understand the fourth polyhedron as a subset of $\RR^{\GT_n\setminus A}$ too. Additionally, we simplify notation by using $\lambda$ for the weight and for the marking simultaneously.

\begin{rem}
	Because of the defining relations of $\V_\lambda$ it is clear that the two inequalities $z_{i,n-1}^\uparrow \leq y_{i,n-1} + y_{i,n} + y_{i+1,n}$ and $z_{i,n-1}^\downarrow \leq y_{i+1,n-1} + y_{i,n} + y_{i+1,n} $ are in fact equivalent. So when checking whether a given point is a tweaked Gelfand-Tsetlin pattern, it is sufficient to just verify one of those inequalities for every $i$. Furthermore, if one of these inequalities happens to be an equality, the other one will be too.
\end{rem}

We will now define an analogue of identity diagrams for these special posets.

\begin{cons}
	Let $(A,\lambda)$ be the pseudo-marking on the tweaked Gelfand-Tsetlin poset $\GT_n$ from \cref{prop:gtorder} and let $x\in \widetilde{GT}(\lambda)$. Let $\mathcal{H}_n$ denote the Hasse diagram of $\GT_n$. The \textbf{tweaked Gelfand-Tsetlin pre-diagram} $\pre\D_{\Dn}^\lambda(x)$ associated to $x$ is the colored directed graph whose nodes are labeled by the elements of the tweaked Gelfand-Tsetlin poset and whose arrows are given by the following construction.
	\begin{enumerate}
		\item Add a {\em black} arrow $p\to q$ if there exists an arrow $p\to q$ in $\mathcal{H}_n$ between the corresponding nodes.
		\item Add a {\em black} arrow $p \to q$ if there exists an opposite arrow $q\to p$ in $\mathcal{H}_n$ between the corresponding nodes and $x_p = x_q$.
		\item For every $1\leq i\leq n-1$ add six (only three if $i=n-1$) {\em red} arrows
		\begin{center}
			\begin{tikzpicture}[>=stealth]
			\draw
			(0,0) node (y1) {$\xi_{i,n-1}$}
			(0,-1) node (y11) {$\xi_{i+1,n-1}$}
			(2,0) node (zup) {$\zeta_{i,n-1}^\uparrow$}
			(2,-1) node (zdown) {$\zeta_{i,n-1}^\downarrow$}			
			(4,0) node (y2) {$\xi_{i,n}$}
			(4,-1) node (y22) {$\xi_{i+1,n}$};
			\draw[->,dashed]
			(y1) edge (zup)
			(zup) edge (y2) 
			(zup) edge (y22)
			(y11) edge (zdown)
			(zdown) edge (y2)
			(zdown) edge (y22);
			\end{tikzpicture}
		\end{center} if $z_{i,n-1}^\uparrow = y_{i,n-1} + y_{i,n} + y_{i+1,n}$ (or equivalently $z_{i,n-1}^\downarrow = y_{i+1,n-1} + y_{i,n} + y_{i+1,n}$).
	\end{enumerate}
\end{cons}

For reasons of readability we will always draw {\em red} arrows as black but dashed arrows. This diagram can be further simplified for our purposes. The reason is the following observation.

\begin{rem}
	Let $x$ be a tweaked Gelfand-Tsetlin pattern. Let $1\leq i\leq n-1$ be an index such that $z_{i,n-1}^\uparrow = y_{i,n-1} + y_{i,n} + y_{i+1,n}$. Then the following implications hold.
	\begin{enumerate}
		\item $y_{i,n-1} = y_{i,n} = y_{i+1,n}$ $\Rightarrow$ $y_{i,n-1} = y_{i,n} = y_{i+1,n} = z_{i,n-1}^\uparrow = 0$.
		\item $z_{i,n-1}^\uparrow = y_{i,n-1}$ $\Rightarrow$ $y_{i,n} = - y_{i+1,n}$.
		\item $z_{i,n-1}^\uparrow = y_{i,n}$ $\Rightarrow$ $y_{i,n-1} = - y_{i+1,n}$.
		\item $z_{i,n-1}^\uparrow = y_{i+1,n}$ $\Rightarrow$ $y_{i,n-1} = - y_{i,n}$.
	\end{enumerate}
	The analogue statements hold true for $z_{i,n-1}^\downarrow$ and $y_{i+1,n-1}$.
\end{rem}

We will use these observations to adapt our pre-diagrams. The goal is to indicate whether two entries must be additive inverses of each other because of an equation of the form $z_{i,n-1}^\uparrow = y_{i,n-1} + y_{i,n} + y_{i+1,n}$ or $z_{i,n-1}^\downarrow = y_{i+1,n-1} + y_{i,n} + y_{i+1,n}$.

\begin{cons}
	Let $x$ be a tweaked Gelfand-Tsetlin pattern and let $\pre\D_{\Dn}^\lambda(x)$ be its tweaked Gelfand-Tsetlin pre-diagram. We will replace {\em red} arrows as follows.
	
	For all triplets of {\em red} arrows 
	
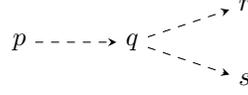
\begin{center}
		\begin{tikzpicture}[>=stealth]
		\draw
		(0,0) node (p) {$p$}
		(1.5,0) node (q) {$q$}
		(3,0.5) node (r) {$r$}
		(3,-0.5) node (s) {$s$};
		\draw[->,dashed]
		(p) edge (q)
		(q) edge (r) 
		(q) edge (s);
		\end{tikzpicture}
	\end{center}
	do the following replacements. (We do not specify whether $r=\xi_{i,n}$ and $s=\xi_{i+1,n}$ or $r=\xi_{i+1,n}$ and $s=\xi_{i,n}$. Both possibilities are allowed!)
	
	If there exist {\em black} arrows $p\to q$, $q\to r$ and $q\to s$, delete all three red arrows but color the four nodes differently. By default we will call every node {\em black}. But in this case we chose to color these nodes {\em white} instead. We will draw them as empty circles, while our normal {\em black} nodes are filled circles.
	
	For all remaining triplets of {\em red} arrows do the following two replacements if possible. If both are applicable to a certain triplet, do only one (it does not matter which one, though the first one is usually preferred due to readability).
	
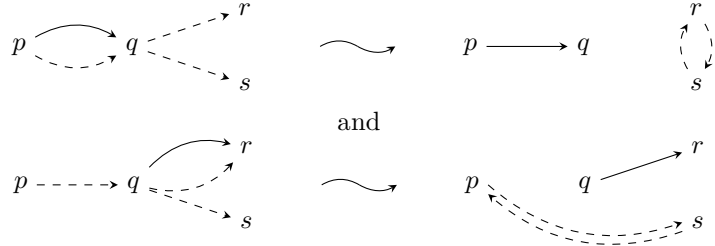
\begin{center}
		\begin{tikzpicture}[>=stealth]
		\draw
		(0,0) node (p) {$p$}
		(1.5,0) node (q) {$q$}
		(3,0.5) node (r) {$r$}
		(3,-0.5) node (s) {$s$};
		\draw[->,dashed]
		(q) edge (r) 
		(q) edge (s);
		\draw[->,dashed,bend right]
		(p) edge (q);
		\draw[->,bend left]
		(p) edge (q);
		\draw
		(4,0) node[minimum size=0em,inner sep=0pt, outer sep=0pt] (l) {}
		(4.5,0) node[minimum size=0em,inner sep=0pt, outer sep=0pt] (m) {}
		(5,0) node[minimum size=0em,inner sep=0pt, outer sep=0pt] (right) {}
		(l) edge[bend left] (m)
		(m) edge[->,bend right] (right);
		\draw
		(6,0) node (pp) {$p$}
		(7.5,0) node (qq) {$q$}
		(9,0.5) node (rr) {$r$}
		(9,-0.5) node (ss) {$s$};
		\draw[->]
		(pp) edge (qq);
		\draw[->,dashed,bend left]
		(rr) edge (ss)
		(ss) edge (rr);				
		\end{tikzpicture}
		
		and
		
		\begin{tikzpicture}[>=stealth]
		\draw
		(0,0) node (p) {$p$}
		(1.5,0) node (q) {$q$}
		(3,0.5) node (r) {$r$}
		(3,-0.5) node (s) {$s$};
		\draw[->,dashed]
		(p) edge (q) 
		(q) edge (s);
		\draw[->,dashed,bend right]
		(q) edge (r);
		\draw[->,bend left]
		(q) edge (r);
		\draw
		(4,0) node[minimum size=0em,inner sep=0pt, outer sep=0pt] (l) {}
		(4.5,0) node[minimum size=0em,inner sep=0pt, outer sep=0pt] (m) {}
		(5,0) node[minimum size=0em,inner sep=0pt, outer sep=0pt] (right) {}
		(l) edge[bend left] (m)
		(m) edge[->,bend right] (right);
		\draw
		(6,0) node (pp) {$p$}
		(7.5,0) node (qq) {$q$}
		(9,0.5) node (rr) {$r$}
		(9,-0.5) node (ss) {$s$};
		\draw[->]
		(qq) edge (rr);
		\draw[->,dashed,bend left]
		(ss.210) to (pp.330);	
		\draw[->,dashed,bend right]
		(pp.360) to (ss.180);
		\end{tikzpicture}
	\end{center}
	
	Our replacement procedure could have produced {\em red} triangles. We will replace those as follows.
	
	
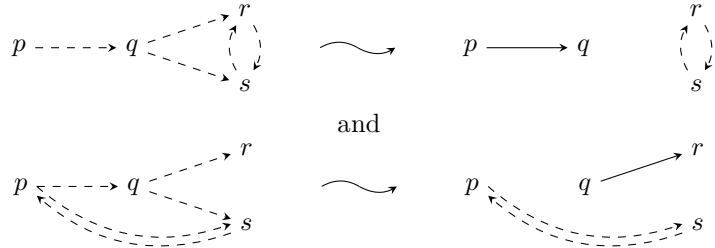
\begin{center}
		\begin{tikzpicture}[>=stealth]
		\draw
		(0,0) node (p) {$p$}
		(1.5,0) node (q) {$q$}
		(3,0.5) node (r) {$r$}
		(3,-0.5) node (s) {$s$};
		\draw[->,dashed]
		(p) edge (q)
		(q) edge (r) 
		(q) edge (s);
		\draw[->,dashed,bend left]
		(r) edge (s)
		(s) edge (r);
		\draw
		(4,0) node[minimum size=0em,inner sep=0pt, outer sep=0pt] (l) {}
		(4.5,0) node[minimum size=0em,inner sep=0pt, outer sep=0pt] (m) {}
		(5,0) node[minimum size=0em,inner sep=0pt, outer sep=0pt] (right) {}
		(l) edge[bend left] (m)
		(m) edge[->,bend right] (right);
		\draw
		(6,0) node (pp) {$p$}
		(7.5,0) node (qq) {$q$}
		(9,0.5) node (rr) {$r$}
		(9,-0.5) node (ss) {$s$};
		\draw[->]
		(pp) edge (qq);
		\draw[->,dashed,bend left]
		(rr) edge (ss)
		(ss) edge (rr);				
		\end{tikzpicture}
		
		and
		
		\begin{tikzpicture}[>=stealth]
		\draw
		(0,0) node (p) {$p$}
		(1.5,0) node (q) {$q$}
		(3,0.5) node (r) {$r$}
		(3,-0.5) node (s) {$s$};
		\draw[->,dashed]
		(p) edge (q) 
		(q) edge (r)
		(q) edge (s);
		\draw[->,dashed,bend left]
		(s.210) to (p.330);	
		\draw[->,dashed,bend right]
		(p.360) to (s.180);
		\draw
		(4,0) node[minimum size=0em,inner sep=0pt, outer sep=0pt] (l) {}
		(4.5,0) node[minimum size=0em,inner sep=0pt, outer sep=0pt] (m) {}
		(5,0) node[minimum size=0em,inner sep=0pt, outer sep=0pt] (right) {}
		(l) edge[bend left] (m)
		(m) edge[->,bend right] (right);
		\draw
		(6,0) node (pp) {$p$}
		(7.5,0) node (qq) {$q$}
		(9,0.5) node (rr) {$r$}
		(9,-0.5) node (ss) {$s$};
		\draw[->]
		(qq) edge (rr);
		\draw[->,dashed,bend left]
		(ss.210) to (pp.330);	
		\draw[->,dashed,bend right]
		(pp.360) to (ss.180);
		\end{tikzpicture}
	\end{center}
	
	Finally, for every $1\leq i\leq n-1$ we will add a pair of {\em black} arrows $\xi_{i,n-1} \leftrightarrows \xi_{i+1,n-1}$ if one of the following two conditions hold.
	\begin{enumerate}
		\item There exist two {\em black} arrows $\zeta_{i,n-1}^\uparrow \to \xi_{i,n}$ and $\zeta_{i,n-1}^\downarrow \to \xi_{i,n}$.
		\item There exist two {\em black} arrows $\zeta_{i,n-1}^\downarrow \to \xi_{i+1,n}$ and $\zeta_{i,n-1}^\downarrow \to \xi_{i+1,n}$.
	\end{enumerate}
	
	The resulting graph is called the \textbf{tweaked Gelfand-Tsetlin diagram} $\D_{\Dn}^\lambda(x)$ of $x$.
\end{cons}

\begin{rem}
	It is clear by construction that $x_p = x_q$ whenever there exists a pair of {\em black} arrows $p\to q$ and $q\to p$. Analogously, $x_p = -x_q$ whenever there exists a pair of {\em red} arrows $p \stackrel{red}{\to} q$ and $q \stackrel{red}{\to}p$. Additionally for every subgraph
	
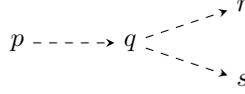
\begin{center}
		\begin{tikzpicture}[>=stealth]
		\draw
		(0,0) node (p) {$p$}
		(1.5,0) node (q) {$q$}
		(3,0.5) node (r) {$r$}
		(3,-0.5) node (s) {$s$};
		\draw[->,dashed]
		(p) edge (q)
		(q) edge (r) 
		(q) edge (s);
		\end{tikzpicture}
	\end{center}
	we have $x_q = x_p+x_r+x_s$. So our construction really visualizes which of the defining inequalities of $\widetilde{GT}(\lambda)$ are actually equalities for the pattern $x$.
	
	Furthermore, $x_p = 0$ for every {\em white} node $p$. The reason is the following. Every white node is part of a quadruplet $\{p,q,r,s\}$ such that $x_p=x_q=x_r=x_s$ and $x_q= x_p+x_r+x_s$ (after proper renaming). But this implies that $x_q = 3x_q$ and hence $0=x_q=x_p=x_q=x_r=x_s$.
\end{rem}

For readability of our drawings, sometimes we do not draw single arrows (we will remember their existence from the positions of the nodes), replace {\em black} double arrows $p\to q$ and $q\to p$ by a straight line and represent {\em red} double arrows $p\stackrel{red}{\to} q$ and $q \stackrel{red}{\to} p$ by a double straight line. We also usually omit {\em red} double arrows between white nodes, since they do not contain any new information. Single {\em red} arrows will be drawn as a dashed line.

\begin{defi}
	Let $x$ be a tweaked Gelfand-Tsetlin pattern and $\D_{\Dn}^\lambda(x)$ its tweaked Gelfand-Tsetlin diagram. A subset $\C$ of nodes of this diagram is called \textbf{connected}, if it is connected via double-{\em black} and double-{\em red} arrows, i.e. for any two nodes $p$ and $q$ in $\C$ there exists a sequence $p_1, \ldots, p_t\in\C$ such that $p=p_1$, $q=p_t$ and for every $i$ there either exist two {\em black} arrows $p_i \to p_{i+1}$ and $p_{i+1} \to p_i$ or there exist two {\em red} arrows $p_{i} \stackrel{red}{\to} p_{i+1}$ and $p_{i+q}\stackrel{red}{\to} p_i$. The sequence $(p_1, \ldots, p_t)$ is called a \textbf{connecting sequence} between $p$ and $q$.
	
	The maximal (with respect to inclusion) connected subsets are called \textbf{connected components}.
\end{defi}

One might be inclined to think that a tweaked Gelfand-Tsetlin pattern is a vertex of the tweaked Gelfand-Tsetlin polytope if and only if each of the connected components in its tweaked Gelfand-Tsetlin diagram contains a marked element. However, this is not true in contrast to types $\An$, $\Bn$ and $\Cn$. There are more possibilities as the following examples show.

\begin{ex}\label{ex:d3}
	Let $G = \SO_6$ and $\lambda = \rho = 2\epsilon_1 + \epsilon_2\in\Lambda^+$. Then the pattern in \cref{fig:d3} is a vertex of $\widetilde{GT}(\lambda)$. However, its tweaked Gelfand-Tsetlin diagram	contains the isolated node $\zeta_{1,2}^\downarrow$.
\end{ex}

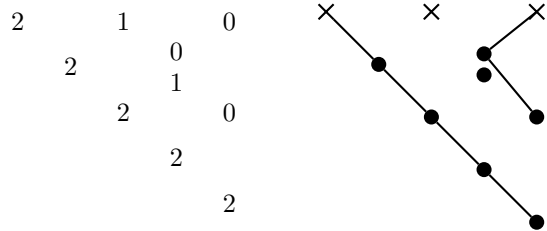
\begin{figure}[htb]\centering
	\caption{Vertex of the tweaked Gelfand-Tsetlin polytope $\widetilde{GT}(2\epsilon_1 + \epsilon_2)$ and its tweaked Gelfand-Tsetlin diagram in type $\mathsf{D}_3$.}\label{fig:d3}
	\begin{tikzpicture}
	\matrix (m) [matrix of nodes,align=center,row sep={4ex,between origins},column sep={2em,between origins},minimum width=2em]
	{
		$2$&&$1$&&$0$\\
		&$2$&&\begin{tabular}{c}$0$\\$1$\end{tabular}&\\
		&&$2$&&$0$\\
		&&&$2$&\\
		&&&&$2$\\
	};
	\end{tikzpicture}\hspace{20pt}\begin{tikzpicture}[every node/.prefix style={circle, fill, minimum size=.2cm, inner sep=0pt, outer sep=0pt},scale=1.4]
	\draw
	(0,0) node[thick,cross out, draw=black, minimum size=.2cm] (1) {}
	(1,0) node[thick,cross out, draw=black, minimum size=.2cm] (2) {}
	(2,0) node[thick,cross out, draw=black, minimum size=.2cm] (3) {}
	(0.5,-0.5) node (4) {}
	(1.5,-0.4) node (5) {}
	(1.5,-0.6) node (55) {}
	(1,-1) node (6) {}
	(2,-1) node (7) {}
	(1.5,-1.5) node (8) {}
	(2,-2) node (9) {};
	\draw[thick]
	(1)
	-- (4)
	-- (6)
	-- (8)
	-- (9);
	\draw[thick]
	(3.center)
	-- (5)
	-- (7);		
	\end{tikzpicture}\end{figure}

\begin{ex}
	Let $G=\SO_8$ and $\lambda=\omega_1+\omega_2+\omega_3+3\omega_4 = 4\epsilon_1+3\epsilon_2+2\epsilon_3-\epsilon_4$. Then the pattern in \cref{fig:d4} is a vertex of $\widetilde{GT}(\lambda)$. We see two things happening in this example. First of all we see a triangular pattern of {\em white} notes. Secondly we see two nodes ($\zeta_{1,3}^\uparrow$ and $\zeta_{1,3}^\downarrow$) that are not connected to any other node\,---\,although they are connected via single {\em red} arrows. But these single arrows do not count as connections in our sense.
\end{ex}

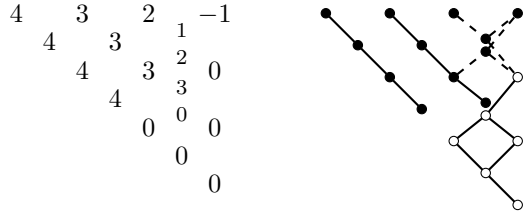
\begin{figure}[htb]\centering
	\caption{Vertex of the tweaked Gelfand-Tsetlin polytope $\widetilde{GT}(4\epsilon_1+3\epsilon_2+2\epsilon_3-\epsilon_4)$ and its tweaked Gelfand-Tsetlin diagram in type $\mathsf{D}_4$.}\label{fig:d4}
	\begin{tikzpicture}
	\matrix (m) [matrix of nodes,align=center,row sep={2.5ex,between origins},column sep={1.25em,between origins},minimum width=2em]
	{
		$4$&&$3$&&$2$&&$-1$\\
		&$4$&&$3$&&\footnotesize\begin{tabular}{c}$1$\\$2$\end{tabular}&\\
		&&$4$&&$3$&&$0$\\
		&&&$4$&&\footnotesize\begin{tabular}{c}$3$\\$0$\end{tabular}&\\
		&&&&$0$&&$0$\\
		&&&&&$0$&\\
		&&&&&&$0$\\			
	};
	\end{tikzpicture}
	\hspace{20pt}		
	\begin{tikzpicture}[every node/.prefix style={draw=black,circle, fill, minimum size=.125cm, inner sep=0pt, outer sep=0pt},scale=.85]
	\draw
	(0,0) node (1) {}
	(1,0) node (2) {}
	(2,0) node (3) {}
	(3,0) node (4) {}
	(0.5,-0.5) node (5) {}
	(1.5,-0.5) node (6) {}
	(2.5,-0.4) node (7) {}
	(2.5,-0.6) node (77) {}
	(1,-1) node (8) {}
	(2,-1) node (9) {}
	(3,-1) node[draw=black,fill=white] (10) {}
	(1.5,-1.5) node (11) {}
	(2.5,-1.4) node (12) {}
	(2.5,-1.6) node[draw=black,fill=white] (1212) {}
	(2,-2) node[draw=black,fill=white] (13) {}
	(3,-2) node[draw=black,fill=white] (14) {}
	(2.5,-2.5) node[draw=black,fill=white] (15) {}
	(3,-3) node[draw=black,fill=white] (16) {};
	\draw[thick]
	(1)
	-- (5)
	-- (8)
	-- (11);
	\draw[thick]
	(2)
	-- (6)
	-- (9)
	-- (12);
	\draw[thick]
	(10)
	-- (1212)
	-- (13)
	-- (15)
	-- (16);
	\draw[thick]
	(1212)
	-- (14)
	-- (15);
	\draw[thick,dashed]
	(3) edge (7)
	(7) edge (4)
	(7) edge (10)
	(9) edge (77) 
	(77) edge (4)
	(77) edge (10);
	\end{tikzpicture}\end{figure}

We will systematize these deviations from the $\An$, $\Bn$ and $\Cn$ cases as follows.

\begin{defi}
	Let $x\in\widetilde{GT}(\lambda)$ be a tweaked Gelfand-Tsetiln pattern and let $\D_{\Dn}^\lambda(x)$ be its tweaked Gelfand-Tsetlin diagram.
	\begin{enumerate}
		\item An \textbf{anomaly} is a triangle of white nodes of the following form.
		
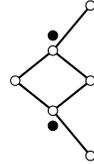
\begin{center}\begin{tikzpicture}[every node/.prefix style={circle, draw=black, minimum size=.125cm, inner sep=0pt, outer sep=0pt},scale=1]
			\draw
			(1,0) node (1) {}
			(0.5,-0.4) node[fill] (2) {}
			(0.5,-0.6) node (22) {}
			(0,-1) node (3) {}
			(1,-1) node (4) {}
			(0.5,-1.4) node (5) {}
			(0.5,-1.6) node[fill] (55) {}
			(1,-2) node (6) {};
			\draw[thick]
			(1)
			-- (22)
			-- (3)
			-- (5)
			-- (6);
			\draw[thick]
			(22)
			-- (4)
			-- (5);
			\end{tikzpicture}\end{center}
		\item A \textbf{single impurity} is a node $\zeta_{i,n-1}^\uparrow$ or $\zeta_{i,n-1}^\downarrow$ with $i< n-1$ that is not connected to any other node but the other node $\zeta_{i,n-1}^\downarrow$ (resp. $\zeta_{i,n-1}^\uparrow$) is connected to another node.
		\item A \textbf{double impurity} is a pair of nodes $(\zeta_{i,n-1}^\uparrow, \zeta_{i,n-1}^\downarrow)$ with $i<n-1$ such that both nodes are not connected to any other node but they are incident to {\em red} arrows.
		\item A \textbf{triviality} is a pair of nodes $(\zeta_{i,n-1}^\uparrow, \zeta_{i,n-1}^\downarrow)$ with $i<n-1$ such that both nodes are not connected to any other node and they are not incident to {\em red} arrows.
	\end{enumerate}
\end{defi}

Notice that we do not allow single impurities to happen at $\zeta_{n-1,n-1}$.

We can now state the complete classification of vertices of tweaked Gelfand-Testlin polytopes, as proved in \cite[Theorem 6.8.16]{S}.

\begin{theorem}\label{thm:tweakedvertices}
	A tweaked Gelfand-Tsetlin pattern is a vertex of the tweaked Gelfand-Tsetlin polytope if and only if each of the connected components of its tweaked Gelfand-Tsetlin diagram contains a marked element, contains an anomaly, is a single impurity or is part of a double impurity.
\end{theorem}

We will not state a proof of this result as it is quite lengthy and technical. The idea is to use \cref{lemma:vertices} and explicitly show that for each $x\in\widetilde{GT}(\lambda)$ there exists a vector $v\neq0$ such that $x\pm v\in\widetilde{GT}(\lambda)$ if and only if $\D_{\Dn}^\lambda(x)$ violates the conditions. The construction of $v$ is very similar to the marked order case, one just has to introduce signs for {\em red} arrows and carefully treat the double nodes. We invite the reader to verify the explicit calculations in \cite[Chapter 6, Section 8]{S}. 

Let us start our return towards string polytopes. For that purpose we introduce the following notation.

\begin{nota}
	For $\lambda = (\lambda_1, \ldots, \lambda_n) \in \RR^n$ let $\Gamma(\lambda)\subseteq\RR$ denote the free abelian group generated by the coefficients $\lambda_1, \ldots, \lambda_n$, i.e.
	\[ \Gamma(\lambda) := \ZZ \lambda_1 + \ldots \ZZ\lambda_n. \]
\end{nota}

This will allow us to state the following crucial result.

\begin{theorem}\label{thm:gamma}
	Let $\lambda\in\RR^n$ and let $x\in\widetilde{GT}(\sum_{i=1}^{n}\lambda_i\epsilon_i)$ be a vertex of the tweaked Gelfand-Tsetlin polytope. Then every coordinate of $x$ lies in $\Gamma(\lambda)$.
\end{theorem}

\begin{proof}
	Let $p\in\GT_n$. We want to calculate $x_p$ via the tweaked Gelfand-Tsetlin diagram $\D_{\Dn}^\lambda(x)$.	
	By \cref{thm:tweakedvertices} we know that $p$ is connected to a marked element, is connected to an anomaly, is a single impurity or is part of a double impurity.
		
	By construction of $\D_{\Dn}^\lambda(x)$ we know that $x_p = \pm \lambda_j$ if $p$ is connected to $\xi_{1,j}$. If $p$ is connected to an anomaly, we know that $x_p = 0$. So in both cases $x_p$ lies in $\{ \pm\lambda_1, \ldots, \pm\lambda_n \}\cup\{0\}\subseteq\Gamma(\lambda)$.
	
	We will only prove the impurity case for $p=\zeta_{i,n}^\uparrow$ because the proof for $\zeta_{i,n-1}^\downarrow$ is completely analogous.	
	If $p=\zeta_{i,n-1}^\uparrow$ is incident to {\em red} arrows, we know that $x_p=z_{i,n-1}^\uparrow$ can be calculated as $z_{i,n-1}^\uparrow = y_{i,n-1} + y_{i,n} + y_{i+1,n}$. Since the latter three coordinates lie in $\Gamma(\lambda)$, the same holds true for $z_{i,n-1}^\uparrow=x_p$.	
	If $p=\zeta_{i,n-1}^\uparrow$ is not incident to {\em red} arrows, we know that $\zeta_{i,n-1}^\downarrow$ cannot be an impurity too. So the value $x_p = z_{i,n-1}^\uparrow$ is uniquely determined by the equation $z_{i,n-1}^\uparrow = z_{i,n-1}^\downarrow + y_{i,n-1}-y_{i+1,n-1}$. Since the latter three coordinates lie in $\Gamma(\lambda)$, the same holds true for $z_{i,n-1}^\uparrow=x_p$.
	This concludes our proof.
\end{proof}

During this proof we have actually shown the following special case.

\begin{cor}\label{cor:12z}
	Let $\lambda\in(\frac{1}{2}+\ZZ)^n$ and let $x\in\widetilde{GT}(\sum_{i=1}^{n}\lambda_i\epsilon_i)$ be a vertex of the tweaked Gelfand-Tsetlin polytope. Then
	\begin{enumerate}
		\item $x_p = 0$ if $p$ is connected to an anomaly,
		\item $x_p \in\frac{1}{2}\ZZ$ if $p$ is a single impurity or part of a double impurity, and
		\item $x_p\in\frac{1}{2}+\ZZ$ for any other $p$.
	\end{enumerate}
\end{cor}

\section{Integrality of Standard String Polytopes in Type $\Dn$}\label{sec:d}

We can now prove the following useful translation.

\begin{cor}\label{cor:tweak}
	Let $G=\SO_{2n}$ and $\lambda=\sum_{i=1}^{n}\lambda_i\epsilon_i\in\Lambda^+$. Let $x$ be a vertex of the standard string polytope $\Q_{\underline{w_0}^\mathrm{std}}(\lambda)$. Then $x$ is a lattice point if and only if \begin{enumerate}
		\item $\lambda_i \in\ZZ$ for all $i$ or
		\item $\lambda_i \in\frac{1}{2}+\ZZ$ for all $i$ and $\D_{\Dn}^\lambda(\widetilde{\phi}_\lambda(x))$ does not contain any anomaly.
	\end{enumerate} 
\end{cor}

\begin{proof}
	By \cref{thm:stringtweak} we know that $x$ is a vertex of $\Q_{\underline{w_0}^\mathrm{std}}(\lambda)$ if and only if $\widetilde{\phi}_\lambda(x)$ is a vertex of $\widetilde{GT}(\lambda)$. Additionally, $x$ will be a lattice point if and only if the coordinates of $\widetilde{\phi}_\lambda(x)$ are either all integers or all in $\frac{1}{2}+\ZZ$. The claim now follows directly from \cref{thm:gamma} if $\lambda_i\in\ZZ$ for all $i$.
	
	Let us consider the other case $\lambda_i\in\frac{1}{2}+\ZZ$ for all $i$. Notice that most of the coordinates will be in $\frac{1}{2}+\ZZ$ by \cref{cor:12z}. If $\D_{\Dn}^\lambda(x)$ contains an anomaly, there will be coordinates equal to zero. Hence $x$ cannot be a lattice point in this case. If there does not exist an anomaly in $\D_{\Dn}^\lambda(x)$, the only nodes whose corresponding coordinate could potentially be not in $\frac{1}{2}+\ZZ$ would be single or double impurities. But the proof of \cref{thm:gamma} shows that these coordinates can be calculated as a sum of three coordinates in $\frac{1}{2}+\ZZ$. Hence they must be in $\frac{1}{2}+\ZZ$ as well.	
\end{proof}

With this result, we can tackle the last remaining case of \cref{thm:main}.

\begin{proof}[Proof of \cref{thm:main} in Type $\Dn$]
	Let $G=\SO_{2n}$ and fix a dominant weight $\lambda=\sum_{i=1}^n\lambda_i\epsilon_i\in\Lambda^+$. If $\langle\lambda,\alpha_{n-1}^\vee\rangle+\langle\lambda,\alpha_n^\vee\rangle$ is an even integer, we know that $\lambda_i\in\ZZ$ for all $1\leq i \leq n$. So $\Q_{\underline{w_0}^\mathrm{std}}(\lambda)$ will be a lattice polytope by \cref{cor:tweak}.
	
	If $\langle\lambda,\alpha_{n-1}^\vee\rangle+\langle\lambda,\alpha_n^\vee\rangle$ is an odd integer, we know that $\lambda_i\in\frac{1}{2}+\ZZ$ for all $1\leq i\leq n$. Since \cref{thm:stringtweak} yields a bijection between the vertices of $\Q_{\underline{w_0}^\mathrm{std}}(\lambda)$ and the vertices of $\widetilde{GT}(\lambda)$, it is necessary and sufficient to find a vertex $x$ of $\widetilde{GT}(\lambda)$ that contains an anomaly in its tweaked Gelfand-Tsetlin diagram $\D_{\Dn}^\lambda(x)$. Then \cref{cor:tweak} implies that the vertex $\widetilde{\phi}_\lambda^{-1}(x)$ of $\Q_{\underline{w_0}^\mathrm{std}}(\lambda)$ will not be a lattice point. 
	Notice that for $n<4$ any tweaked Gelfand-Tsetlin diagram cannot contain an anomaly if $\lambda_n\neq 0$. So for small ranks all standard string polytopes must be lattice polytopes.
	
	For $n\geq 4$ consider the pattern in \cref{fig:dn}. We can verify quite easily that this pattern is indeed a tweaked Gelfand-Tsetlin pattern for $\lambda$. The only nontrivial (in)equality to verify is $\lambda_{n-1}+\lambda_n\geq 0$. But this is true for any dominant integral weight in type $\Dn$. The tweaked Gelfand-Tsetlin diagram of this pattern is drawn in \cref{fig:dndiagram}. We see that every connected component contains a marked element or an anomaly. Hence this pattern is indeed a vertex of the tweaked Gelfand-Tsetlin polytope containing an anomaly in its tweaked Gelfand-Tsetlin diagram. This concludes the proof.
\end{proof}

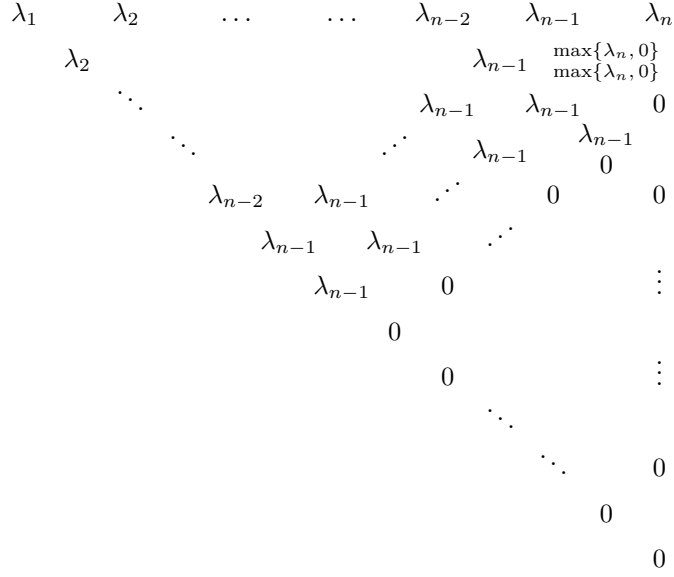
\begin{figure}[htb]\centering
	\caption{A special tweaked Gelfand-Tsetlin pattern in type $\Dn$ for an arbitrary weight $\lambda$.}\label{fig:dn}
	\begin{tikzpicture}
	\matrix (m) [matrix of nodes,align=center,row sep={4ex,between origins},column sep={2em,between origins},minimum width=2em]
	{
		$\lambda_1$&&$\lambda_2$ &&$\ldots$&&$\ldots$&&$\lambda_{n-2}$ && $\lambda_{n-1}$&&$\lambda_n$\\
		&$\lambda_2$&&&&&&&&$\lambda_{n-1}$&&\scriptsize\begin{tabular}{c}$\max\{\lambda_n,0\}$\\$\max\{\lambda_n,0\}$\end{tabular}&\\
		&&$\ddots$&&&&&&$\lambda_{n-1}$&&$\lambda_{n-1}$&&$0$\\
		&&&$\ddots$&&&&$\udots$&&$\lambda_{n-1}$&&\begin{tabular}{c}$\lambda_{n-1}$\\$0$\end{tabular}&\\
		&&&&$\lambda_{n-2}$&&$\lambda_{n-1}$&&$\udots$&&$0$&&$0$\\
		&&&&&$\lambda_{n-1}$&&$\lambda_{n-1}$&&$\udots$&&&\\
		&&&&&&$\lambda_{n-1}$&&$0$&&&&$\vdots$\\
		&&&&&&&$0$&&&&&\\
		&&&&&&&&$0$&&&&$\vdots$\\
		&&&&&&&&&$\ddots$&&&\\
		&&&&&&&&&&$\ddots$&&$0$\\
		&&&&&&&&&&&$0$&\\
		&&&&&&&&&&&&$0$\\
	};
	\end{tikzpicture}\end{figure}

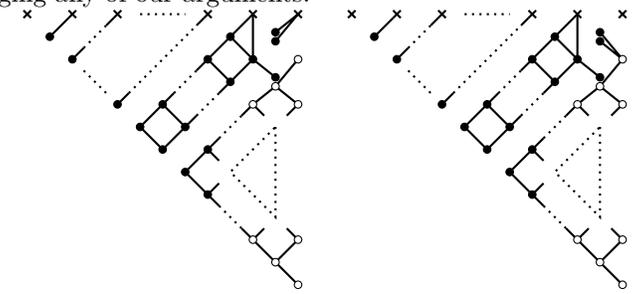
\begin{figure}\centering
	\caption{Tweaked Gelfand-Tsetlin diagrams of the tweaked Gelfand-Tsetlin pattern from \cref{fig:dn}. Left hand side for $\lambda_n>0$, right hand side for $\lambda_n<0$. We decided to only draw the case where $\lambda_1>\lambda_2>\ldots>\lambda_{n-1}>\lambda_n\neq 0$ since any equality between the coefficients of $\lambda$ would only result in more pairs of {\em black} arrows connecting formerly disjoint connected components, thus not changing any of our arguments.}\label{fig:dndiagram}
	
	\begin{tikzpicture}[every node/.prefix style={draw=black, circle, fill, minimum size=.1cm, inner sep=0pt, outer sep=0pt},scale=.6]
	\draw
	(0,0) node[cross out,thick] (11) {}
	(1,0) node[cross out,thick] (12) {}
	(2,0) node[cross out,thick] (13) {}
	(4,0) node[cross out,thick] (15) {}
	(5,0) node[cross out,thick] (16) {}
	(6,0) node[cross out,thick] (17) {}
	(0.5,-0.5) node (21) {}
	(4.5,-0.5) node (25) {}
	(5.5,-0.4) node (26u) {}
	(5.5,-0.6) node (26d) {}
	(1,-1) node (32) {}
	(4,-1) node (35) {}
	(5,-1) node (36) {}
	(6,-1) node[draw=black,fill=white] (37) {}
	(4.5,-1.5) node (45) {}
	(5.5,-1.4) node (46u) {}
	(5.5,-1.6) node[draw=black,fill=white] (46d) {}
	(2,-2) node (53) {}
	(3,-2) node (54) {}
	(5,-2) node[draw=black,fill=white] (56) {}
	(6,-2) node[draw=black,fill=white] (57) {}
	(2.5,-2.5) node (63) {}
	(3.5,-2.5) node (64) {}
	(3,-3) node (74) {}
	(4,-3) node (75) {}
	(3.5,-3.5) node (84) {}
	(4,-4) node (95) {}
	(5,-5) node[draw=black,fill=white] (116) {}
	(6,-5) node[draw=black,fill=white] (117) {}
	(5.5,-5.5) node[draw=black,fill=white] (126) {}
	(6,-6) node[draw=black,fill=white] (137) {}
	;
	\draw[thick]
	(16.center) -- (36)
	(12) -- (21)
	(13) -- (1.75,-0.25)
	(32) -- (1.25,-0.75)
	(15) -- (3.75,-0.25)
	(53) --(2.25,-1.75)
	(16) -- (25) -- (35) -- (3.75,-1.25)
	(63) -- (54) -- (3.25,-1.75)
	(36) -- (45) -- (4.25,-1.75)
	(74) -- (64) -- (3.75,-2.25)
	(37) -- (46d) -- (56) -- (4.75,-2.25)
	(84) -- (75) -- (4.25,-2.75)
	(57) -- (5.75,-2.25)
	(95) -- (4.25,-3.75)
	(116) -- (5.25,-4.75)
	(126) -- (117)
	(63) -- (74)
	(54) -- (64)
	(35) -- (45) 
	(25) -- (36)
	(36) -- (46u)
	(26u) -- (17.center)
	(26d) -- (17.center)
	(46d) -- (57)
	(56) -- (5.25,-2.25)
	(75) -- (4.25,-3.25)
	(84) -- (95)
	(95) -- (4.25,-4.25)
	(116) -- (4.75,-4.75)
	(116) -- (126)
	(117) -- (5.75,-4.75)
	(126) -- (137)		
	;
	\draw[thick,dotted]
	(2.5,0) -- (3.5,0)
	(1.25,-0.75) -- (1.75,-0.25)
	(1.25,-1.25) -- (1.75,-1.75)
	(2.25,-1.75) -- (3.75,-0.25)
	(3.25,-1.75) -- (3.75,-1.25)
	(3.75,-2.25) -- (4.25,-1.75)
	(4.25,-2.75) -- (4.75,-2.25)
	(4.25,-4.25) -- (4.75,-4.75)
	(5.5,-2.5) -- (5.5,-4.5) --(4.5,-3.5) -- cycle
	;
	\end{tikzpicture}
	\hspace{10pt}
	\begin{tikzpicture}[every node/.prefix style={draw=black,circle, fill, minimum size=.1cm, inner sep=0pt, outer sep=0pt},scale=.6]
	\draw
	(0,0) node[cross out,thick] (11) {}
	(1,0) node[cross out,thick] (12) {}
	(2,0) node[cross out,thick] (13) {}
	(4,0) node[cross out,thick] (15) {}
	(5,0) node[cross out,thick] (16) {}
	(6,0) node[cross out,thick] (17) {}
	(0.5,-0.5) node (21) {}
	(4.5,-0.5) node (25) {}
	(5.5,-0.4) node (26u) {}
	(5.5,-0.6) node (26d) {}
	(1,-1) node (32) {}
	(4,-1) node (35) {}
	(5,-1) node (36) {}
	(6,-1) node[draw=black,fill=white] (37) {}
	(4.5,-1.5) node (45) {}
	(5.5,-1.4) node (46u) {}
	(5.5,-1.6) node[draw=black,fill=white] (46d) {}
	(2,-2) node (53) {}
	(3,-2) node (54) {}
	(5,-2) node[draw=black,fill=white] (56) {}
	(6,-2) node[draw=black,fill=white] (57) {}
	(2.5,-2.5) node (63) {}
	(3.5,-2.5) node (64) {}
	(3,-3) node (74) {}
	(4,-3) node (75) {}
	(3.5,-3.5) node (84) {}
	(4,-4) node (95) {}
	(5,-5) node[draw=black,fill=white] (116) {}
	(6,-5) node[draw=black,fill=white] (117) {}
	(5.5,-5.5) node[draw=black,fill=white] (126) {}
	(6,-6) node[draw=black,fill=white] (137) {}
	;
	\draw[thick]
	(16.center) -- (36)
	(12) -- (21)
	(13) -- (1.75,-0.25)
	(32) -- (1.25,-0.75)
	(15) -- (3.75,-0.25)
	(53) --(2.25,-1.75)
	(16) -- (25) -- (35) -- (3.75,-1.25)
	(63) -- (54) -- (3.25,-1.75)
	(36) -- (45) -- (4.25,-1.75)
	(74) -- (64) -- (3.75,-2.25)
	(37) -- (46d) -- (56) -- (4.75,-2.25)
	(84) -- (75) -- (4.25,-2.75)
	(57) -- (5.75,-2.25)
	(95) -- (4.25,-3.75)
	(116) -- (5.25,-4.75)
	(126) -- (117)
	(63) -- (74)
	(54) -- (64)
	(35) -- (45) 
	(25) -- (36)
	(36) -- (46u)
	(26u) -- (37)
	(26d) -- (37)
	(46d) -- (57)
	(56) -- (5.25,-2.25)
	(75) -- (4.25,-3.25)
	(84) -- (95)
	(95) -- (4.25,-4.25)
	(116) -- (4.75,-4.75)
	(116) -- (126)
	(117) -- (5.75,-4.75)
	(126) -- (137)		
	;
	\draw[thick,dotted]
	(2.5,0) -- (3.5,0)
	(1.25,-0.75) -- (1.75,-0.25)
	(1.25,-1.25) -- (1.75,-1.75)
	(2.25,-1.75) -- (3.75,-0.25)
	(3.25,-1.75) -- (3.75,-1.25)
	(3.75,-2.25) -- (4.25,-1.75)
	(4.25,-2.75) -- (4.75,-2.25)
	(4.25,-4.25) -- (4.75,-4.75)
	(5.5,-2.5) -- (5.5,-4.5) --(4.5,-3.5) -- cycle
	;
	\end{tikzpicture}
	
\end{figure}

\section{Gorenstein Fano Toric Degenerations}\label{sec:applications}

We will conclude our paper with some remarks and consequences. 

\begin{rem}
	It should be noted that morally it is absolutely not clear why the string polytope should know {\em anything} about the representations of the underlying algebraic group. Firstly, its definition and many of its explicit descriptions are done purely from the perspective of the Lie algebra (see \cite{L} and \cite{BZ2}). Secondly, we have already seen in \cite[Example 5.5]{S2} that this connection does not hold for arbitrary reduced decompositions. So the connection between the standard reduced decomposition and representations of the algebraic group remains mysterious.
\end{rem}

We would now like to impose some geometric meaning to our result. Batyrev proved in \cite[Proposition 2.2.23]{B} that reflexive polytopes are in one-to-one correspondence to Gorenstein Fano toric varieties. So we want to understand whether a given string polytope is not only integral but reflexive.
In \cite{S2} we proved the following fact.

\begin{theorem}\label{thm:dual}
	If the valuation semigroup $\Gamma(\lambda)$ associated to a partial flag variety $G/P$ via the $P$-regular dominant integral weight $\lambda$ and full-rank valuation $v$ is finitely generated and saturated, the following properties of the Newton-Okounkov body $\Delta(\lambda)$ are equivalent.
	\begin{enumerate}[label=(\roman*)]
		\item $\L_\lambda$ is the anticanonical line bundle over $G/P$.
		\item $\Delta(\lambda)$ contains exactly one lattice point $p_{\lambda}$ in its interior.
	\end{enumerate}
	Furthermore, in this case the polar dual\footnote{We always refer to the {\em polar dual} defined as $S^* := \left\lbrace y \in \RR^N \mm \langle x,y \rangle \leq 1 \text{ for all } x \in S\right\rbrace$ for an arbitrary set $S \subseteq \RR^N$. If $S$ is a polytope with the origin in its interior, then $S^*$ is a polytope.} of the translated Newton-Okounkov body $\Delta(\lambda) - p_{\lambda}$ is a lattice polytope.
\end{theorem}

Using this observation we can give a very precise criterion on the reflexivity of string polytopes.

\begin{theorem*}[\cref{cor:reflexivestring}]
	Let $G$ be a complex classical group and $\lambda$ a dominant integral weight. The standard string polytope $\Q_{\underline{w_0}^\mathrm{std}}(\lambda)$ (in the sense of \cite{L}) is a reflexive polytope after translation by a lattice vector if and only if $\lambda$ is the weight of the anticanonical line bundle over some partial flag variety $G/P$.
\end{theorem*}

\begin{proof}
	Let us denote the weight of the anticanonical line bundle over $G/P$ by $\lambda_{G/P}$.
	
	Notice that this anticanonical bundle can be realized as the highest wedge power of the cotangent bundle over $G/P$, i.e.
	\[ \L_{\lambda_{G/P}} =  \bigwedge^{\dim G/P} (\g/\mathfrak{p})^*.\]
	From this it is clear that $V(\lambda_{G/P})^* \simeq H^0(G/P, \L_{\lambda_{G/P}})$ carries the structure of a $G$-representation. So by \cref{thm:main} we know that $\Q_{\underline{w_0}^\mathrm{std}}(\lambda_{G/P})$ must be a lattice polytope.
	
	The claim now follows directly from \cref{thm:dual} since Kaveh showed in \cite[Theorem 1]{Ka} that every string polytope can be realized as a Newton-Okounkov for a valuation with sufficient properties.
\end{proof}

These results yield the following geometric interpretation.

\begin{theorem*}[\cref{thm:flag}]
	Let $G$ be a complex classical group. Then every partial flag variety admits a flat projective degeneration to a toric Gorenstein Fano variety.
\end{theorem*}

\begin{proof} 
	By a result of Alexeev and Brion (see \cite[Theorem 3.2]{AB}) we know that each flag variety $G/P$ admits a flat projective degeneration to the toric variety associated to the string polytope $\Q_{\underline{w_0}}(\lambda)$ for any reduced decomposition $\underline{w_0}$ and any dominant integral weight $\lambda$. Choosing $\underline{w_0}$ as the standard reduced decomposition and choosing $\lambda$ as the weight of the anticanonical line bundle over $G/P$, we see that the resulting string polytope will be reflexive by \cref{cor:reflexivestring}. Hence the limit toric variety will be Gorenstein Fano by Batyrev's result \cite[Proposition 2.2.23]{B}.
\end{proof}

We will finally remark that using our diagrammatic approach, it is possible to find the unique interior lattice point from \cref{thm:dual}. Basically, we just have to find the Gelfand-Tsetlin pattern with the least amount of {\em black} double arrows possible in its identity diagram. We will show this idea in the following examples.

\begin{ex}
	Let us first consider the case $G=\SL_3$ and the full flag variety $G/B$. Then the anticanonical weight is given by $\lambda_{G/B} = 2\rho$. From \cref{cor:reflexivestring} we know that the standard string polytope $\Q_{\underline{w_0}^\mathrm{std}}(2\rho)$ will be translated to a reflexive polytope. Hence it contains a unique lattice point in its interior. We can find this lattice point by finding the unique Gelfand-Tsetlin pattern of type $2\rho$ such that every entry is neither equal to its left upper neighbor nor its right upper neighbor. This pattern is
	\begin{center}\begin{tikzpicture}
		\matrix (m) [matrix of nodes,align=center,row sep={4ex,between origins},column sep={2.5em,between origins},minimum width=2em]
		{$4$&&$2$&&$0$\\
			&$3$&&$1$&\\
			&&$2$&&\\
		};
		\end{tikzpicture}\end{center}
	and its preimage under Littelmann's bijection $\phi_{2\rho}$ is the point $(1,2,1)$.
\end{ex}

\begin{ex}
	For $G=\SO_5$ and the partial flag variety $G/P(\alpha_1)$, the anticanonical weight is given by $\lambda_{G/P(\alpha_1)}=4\omega_2 = 2\epsilon_1+2\epsilon_2$. Since the standard string polytope $\Q_{\underline{w_0}^\mathrm{std}}(4\omega_2)$ is not full-dimensional, we will only find points in its relative interior. We can also see this diagrammatically. Every identity diagram in this setting will have three connected nodes in the upper left corner since $2=\lambda_1\geq z_{1,1} \geq \lambda_2 = 2$. But if we re-embed the string polytope such that it is indeed full-dimensional, it will contain a unique interior lattice point. This point is associated to the pattern
	\begin{center}\begin{tikzpicture}
		\matrix (m) [matrix of nodes,align=center,row sep={4ex,between origins},column sep={2.5em,between origins},minimum width=2em]
		{$2$&&$2$&&$0$\\
			&$2$&&$\frac{1}{2}$&\\
			&&$1$&&$0$\\
			&&&$\frac{1}{2}$&\\
			&&&&$0$\\
		};
		\end{tikzpicture}\end{center}	
	whose preimage under Littelmann's bijection $\phi_{4\omega_2}$ is the point $(1,2,3,0)$.
\end{ex}

\begin{rem}
	For the full flag variety, let $p_{\mathsf{X}_n}$ denote the unique interior lattice point of the anticanonical standard string polytope $\Q_{\underline{w_0}^\mathrm{std}}(\lambda_{G/B})$. Fujita and Higashitani have calculated $p_{\An}$ and $p_{\Bn}$, as well as $p_{\EE}$, $p_{\EEE}$ and $_{\EEEE}$ in \cite[Example 4.6]{FH}. This list can be completed as follows.
	\begin{align*}
		p_{\An} =\,\, &(1, 2,1, 3,2,1, \ldots, n, n-1, \ldots, 2,1)\\
		p_{\Bn} =\,\, &(1, 2,3,1, 4,3,5,2,1, \ldots, 2n-2, 2n-3, \ldots, n, 2n-1, n-1, \ldots, 2,1)\\
		p_{\Cn} =\,\, &(1, 3,2,1, 5,4,3,2,1, \ldots, 2n-1, 2n-2, \ldots, n+1, n, n-1, \ldots, 2,1)\\
		p_{\Dn} =\,\, &(1,1, 3,2,2,1, 5,4,3,3,2,1, \ldots,\\&\hspace{20pt} 2n-3, 2n-4, \ldots, n, n-1, n-1, n-2, \ldots, 2,1)\\
		p_{\EE} =\,\, &(p_{\mathsf{D}_5}, 11,10,9,8,8,7,7,6,6,5,4,5,4,3,2,1)\\
		p_{\EEE} =\,\, &(p_{\mathsf{E}_6}, 17,16,15,14,13,13, 12,12,11,11,\\&\hspace{20pt}10,9,10,9,8,7,6,9, 8,7,6,5,5,4, 3,2,1)\\
		p_{\EEEE} =\,\, &(p_{\mathsf{E}_7}, 28,27,26,25,24,23,23,22,22,21,21,20,19,20,19,\\&\hspace{20pt}18,17,16,19,18,17,16,15,15,14,13,12,11,29,18,17,16,\\&\hspace{20pt}15,14,14,13,13,12,12,11,10,11,10,9,8,7,6,6,5,4,3,2,1),\\
		p_{\FF} =\,\,& (1, 2,3,1, 4,3,5,2,1, 10,9,8,7,7,6, 5,11,6,5,4,4, 3,2,1 )\\
		p_{\GG} =\,\,& (1,2,5,3,4,1)\hspace{10pt} \text{ and } \hspace{10pt} p_{\GG}' := (1,4,3,5,2,1).
	\end{align*}
	Here $p_{\GG}$ corresponds to $\underline{w_0}=s_1s_2s_1s_2s_1s_2$  ($\alpha_1$ being the short root), whereas $p_{\GG}'$ corresponds to the other reduced decomposition.
\end{rem}

\end{document}